\documentclass[smallextended,numbook,runningheads]{svjour3}     % onecolumn (second format)
\smartqed  % flush right qed marks, e.g. at end of proof
\usepackage{lipsum}
\usepackage{amsfonts}
\usepackage{graphicx}
\usepackage{epstopdf}
\usepackage{algorithmic}
\usepackage{comment}
\ifpdf
  \DeclareGraphicsExtensions{.eps,.pdf,.png,.jpg}
\else
  \DeclareGraphicsExtensions{.eps}
\fi

\author{Daniel Appel\"{o}\thanks{
Department of Computational Mathematics, Science \& Engineering and 
Department of Mathematics, Michigan State University, East Lansing MI 48824, USA. \email{appeloda@msu.edu}} 
\and Lu Zhang\thanks{Department of Applied Physics and Applied Mathematics, Columbia University, New York, NY 10027, USA.   \email{lz2784@columbia.edu}}
\and Thomas Hagstrom\thanks{Department of Mathematics, Southern Methodist University, Dallas, TX 75275, USA. \email{thagstrom@smu.edu}} 
\and Fengyan Li\thanks{Department of Mathematical Sciences, Rensselaer Polytechnic Institute, Troy, NY 12180, USA. \email{lif@rpi.edu}}}

\usepackage{amsopn}

\usepackage{color,url,enumitem,graphicx}
\usepackage{epic,comment}
\usepackage{amssymb,amsmath}

\newcommand{\ba}{\begin{array}}
\newcommand{\ea}{\end{array}}
\newcommand{\be}{\begin{equation}}
\newcommand{\ee}{\end{equation}}
\newcommand{\ben}{\begin{equation*}}
\newcommand{\een}{\end{equation*}}
\newcommand{\bd}{\begin{displaymath}}
\newcommand{\ed}{\end{displaymath}}
\newcommand{\bi}{\begin{itemize}}
\newcommand{\ei}{\end{itemize}}
\newcommand{\bn}{\begin{enumerate}}
\newcommand{\en}{\end{enumerate}}
\newcommand{\pa}{\partial}
\newcommand{\f}{\frac}
\newcommand{\ci}{\cite}

\newtheorem{truth}{Theorem}

\newcommand{\mH}{\mathcal H}
\newcommand{\mG}{\mathcal G}
\newcommand{\mL}{\mathcal L}

\newcommand{\dbl}{[[}
\newcommand{\dbr}{]]}
\newcommand{\dsl}{\{\{}
\newcommand{\dsr}{\}\}}

\usepackage{mathptmx}      % use Times fonts if available on your TeX system
\journalname{BIT}
\begin{document}

\title{An Energy-Based Discontinuous Galerkin Method with Tame CFL Numbers for the Wave Equation\thanks{This work was partially supported by NSF Grants DMS-1913076, DMS-2012296, DMS-1719942 and DMS-1913072. Any opinions, findings, conclusions or recommendations expressed in this material are those of the authors and do not necessarily reflect the views of the National Science Foundation.}}

\titlerunning{Energy-Based DG with Tame CFL}        % if too long for running head

\author{Daniel Appel\"{o}         \and
        Lu Zhang \and
        Thomas Hagstrom \and
        Fengyan Li
}


\institute{Daniel Appel\"{o} \at 
    Department of Computational Mathematics, Science \& Engineering and 
    Department of Mathematics, Michigan State University, East Lansing MI 48824, USA. \email{appeloda@msu.edu} \and
    Lu Zhang \at 
    Department of Applied Physics and Applied Mathematics, Columbia University, New York, NY 10027, USA.   \email{lz2784@columbia.edu}
    \and
    Thomas Hagstrom
    \at Department of Mathematics, Southern Methodist University, Dallas, TX 75275, USA. \email{thagstrom@smu.edu} \and 
    Fengyan Li \at Department of Mathematical Sciences, Rensselaer Polytechnic Institute, Troy, NY 12180, USA. \email{lif@rpi.edu}
    }

\date{Received: date / Accepted: date}

\maketitle

\begin{abstract}
We extend and analyze the energy-based discontinuous Galerkin method for second order wave equations on staggered and structured meshes. By combining spatial staggering with local time-stepping near boundaries, the method overcomes the typical numerical stiffness associated with high order piecewise polynomial approximations. In one space dimension with periodic boundary conditions and suitably chosen numerical fluxes, we prove bounds on the spatial operators that establish stability for CFL numbers $c \f {\Delta t}{h} < C$ independent of order when stability-enhanced explicit time-stepping schemes of matching order are used. For problems on bounded domains and in higher dimensions we demonstrate numerically that one can march explicitly with large time steps at high order temporal and spatial accuracy.  
%Include keywords and mathematical subject classification numbers as needed.
\end{abstract}

\section{Introduction}

Discontinuous Galerkin methods \cite{cockburn1989tvb,HesthavenWarburton02} have emerged as one of the most popular discretization techniques for simulating physical and engineering phenomena including various linear and nonlinear wave models. Discontinuous Galerkin methods 
have excellent dispersive properties, are geometrically flexible, do not have a global mass matrix, can be implemented at any order of accuracy, and, being Galerkin methods, have robust stability properties.    

Although discontinuous Galerkin methods are spectrally convergent with the order $q$ of the approximation, very high order methods, say $q>6$, are seldom used in practice. The primary reason for this is that the spectral radius of the discrete spatial derivative operator grows as $q^2/h$, where $h$ is an element length scale. This rapidly growing numerical stiffness forces the use of excessively small time steps, effectively prohibiting the use of very high order methods. The source of this numerical stiffness is the approximation by polynomials which must be sampled throughout an element. Heuristically this can be understood by comparing a wave $w=e^{i q x}$ and its $q$ times larger derivative $w_x = iq w$ with a typical orthogonal polynomial, say a Chebyshev polynomial, $T_q(x) =\cos (q \cos^{-1}(x))$ and its derivative $T_q'(x)$. Clearly $| T_q(x) | \le 1,$ for $ |x| \le 1$ and for all $q$, as for the wave, but the derivative $|T_q'(\pm 1)| = q^2$, is $q$ times larger at the edges. 

This numerical stiffness is particularly troublesome for linear wave propagation problems where solutions typically remain smooth throughout the computation and thus favor very high order spatial discretizations. Fortunately this numerical stiffness can be circumvented in several ways, for example by allowing the polynomial approximations to the solution to spread out over many elements as in the traditional finite difference methods or as in  the more recent Galerkin difference methods \cite{BANKS2016310}, or by  only sampling the derivatives at the cell center as in Hermite methods \cite{secondHermite}. 

It is also possible to remove this numerical stiffness within the discontinuous Galerkin framework either by co-volume filtering as proposed in \cite{TAMECFL} or by carrying two approximate solutions on staggered grids as in central discontinuous Galerkin methods \cite{liu20082}. Central discontinuous Galerkin methods combine features of discontinuous Galerkin methods and central schemes \cite{nessyahu1990non}, and they are shown  in \cite{liu20082}, via Fourier analysis, to allow a larger time step size than standard upwind discontinuous Galerkin methods when applied to the linear advection equation. In \cite{fli} these results were established quantitatively for upwind discontinuous Galerkin methods and central discontinuous Galerkin methods by estimating the dependence of the operator bound of the respective spatial discretizations on the approximation order. A serious drawback with the co-volume approach \ci{TAMECFL} and the central discontinuous Galerkin approach \ci{liu20082} is that they must carry two copies of the solution hence with memory usage and  computational cost per right hand side evaluation doubled.  

In this paper we present an alternative method that can be time-marched at {\bf very high order of accuracy and with an explicit time discretization and $\mathcal{O}(1)$ CFL}. Our method is a staggered version of the energy-based discontinuous Galerkin method \cite{Upwind2}. Our method does not require any additional copies of the solution vector and thus has the same  
memory cost as the original method in \ci{Upwind2} but can take much larger timesteps. 

We prove in one dimension and with periodic boundary conditions that this staggered energy-based discontinuous Galerkin method results in a semi-discrete-in-space operator whose norm  grows {\em linearly} with the order of the method. Precisely, in Theorem \ref{thm:staggerbound} we establish a bound for the spatial operator $\mathcal{L}_c$:
\begin{displaymath}
\| \mathcal{L}_c \| \leq C \frac {q}{h} ,
\end{displaymath}
where $h$ is the element width and $q$ is the polynomial degree. 
This, in combination with the  Kreiss-Wu theory \cite{kreiss1993stability}, indicates  that  the Courant-Friedrich-Levy (CFL) number is constant independent of order of accuracy as long as high order locally stable time stepping methods with large stability domains are applied. Such time-stepping methods can be constructed at arbitrary order by adding additional stages to enhance the stability of standard methods; see, for example, \cite{JolyRodriguezLeapFrog} where stability-enhanced leap-frog schemes are proven to exist at all orders and optimized at orders up to sixteen. 

At physical boundaries it is no longer possible to stagger the mesh and the CFL constraint becomes order dependent again. As long as the bulk of the problem can be meshed by a rectilinear mesh this can easily be remedied in any dimension by the use of local timestepping in elements near the boundary. Here we use the local timestepping methods of Diaz and Grote \cite{diaz2009energy} and show in numerical experiments in one and two dimensions that this approach allows us to retain the large time steps in the interior. The resulting method, while having some additional computational overhead near boundaries, asymptotically has the same computational complexity as the staggered method for the periodic case. 

The two dimensional examples we consider below are proof-of-concept computations in square geometries but we emphasize that a more sophisticated (than the one we have used for the results in this paper) implementation could be very efficient for meshes of the type that is displayed in Figure  \ref{fig:hybrid_mesh}, and that extensions to three spatial dimensions are straightforward. An example of problems of this type is the simulation of underwater acoustics with bathymetry.  

The rest of the paper is organized as follows. In section \ref{sec:method} we review the formulation of energy-based DG methods for the scalar wave equation and extend them to staggered,
structured meshes. In section \ref{sec:operatorbounds} we establish bounds on the norm of the spatial operator in one space dimension and with periodic boundary conditions. In section \ref{sec:timestepping} we briefly discuss our time-stepping schemes and the corrections needed to maintain large time steps in the presence of boundaries. Lastly, in section \ref{sec:experiment} we demonstrate the accuracy and stability of the method in one and two space dimensions by means of numerical experiments. 

\begin{figure}[tb]
	\begin{center} 
		\includegraphics[width=0.95\textwidth]{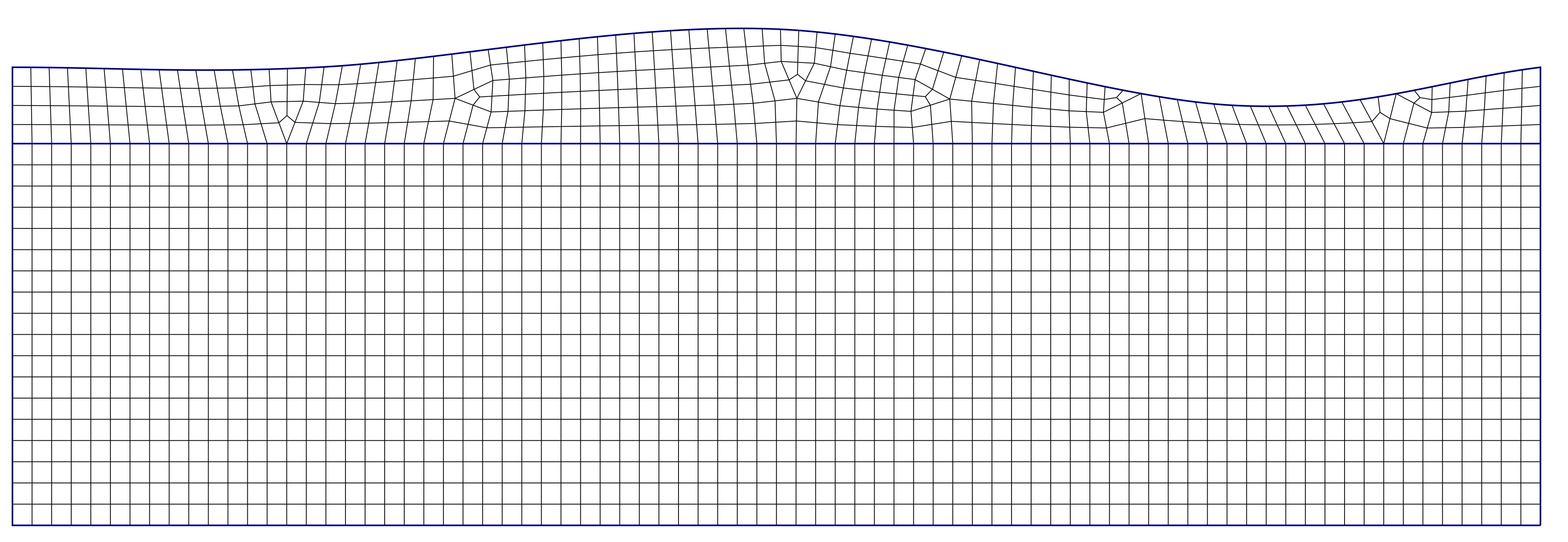}
\caption{The method presented here can handle problems on meshes as the one above. \label{fig:hybrid_mesh}}
	\end{center}
\end{figure}

\section{Energy-Based Discontinuous Galerkin Method for the Wave Equation} \label{sec:method}
We consider the scalar wave equation written as a first order system in time
\begin{eqnarray}
&&\f{\pa u(x,t)}{\pa t} = v(x,t), \label{eq:second1} \\
&& \f{\pa v(x,t)}{\pa t}= \nabla \cdot (c^2(x) \nabla u(x,t)) + f(x,t),  \label{eq:second2}
\end{eqnarray}
on the domain 
\[
x = (x_1,\ldots,x_d) \in \Omega \subset \mathbb{R}^d, \, t>0,
\]
with initial conditions  
\be
u(x,0) = g(x), \ \ \f{\pa u}{\pa t }(x,0) = v(x,0) = h(x),
\ee
and boundary conditions  
\be
\gamma \f{\pa u}{\pa t } + \kappa c \, \vec{n} \cdot \nabla u  = 0,  \ \ 
x
\in \pa \Omega.  \label{eq:bc}
\ee
Here $c=c(x)$ is the speed of sound and $\vec{n}$ is the outward pointing unit normal. For the boundary conditions we assume the normalization $\gamma^2+\kappa^2 = 1$ and that $\gamma,\kappa \ge 0$. Then the choice $\kappa = 0$ corresponds to a homogeneous Dirichlet boundary condition on {$\f{\pa u}{\pa t }$}  
and  $\gamma = 0$ corresponds to a homogeneous Neumann boundary condition.  Any choice with $\gamma \kappa$ being  positive will dissipate the energy of the system and can be thought of as a low order non-reflecting boundary condition. 

The energy associated with the scalar wave equation is 
\begin{equation}
\mathcal{E}(t) = \frac{1}{2} \int_\Omega \left(\frac{\pa u}{\pa t}\right)^2 + c^2(x) |\nabla u|^2  {dx,}
\end{equation}
and it is a discrete version of this energy that our energy-based discontinuous Galerkin method is built from.   

We now present the non-staggered and staggered formulations of the method. A more thorough analysis of the non-staggered method can be found in \cite{Upwind2}, but we include it here to illustrate the differences between the two formulations. The essential new idea in the energy-based method is to enforce equation (\ref{eq:second1}) {\em weakly} with a nonstandard test function; see equations (\ref{var1}) and (\ref{var1_s}) below. With this choice we can establish energy estimates without the need for mesh-dependent penalty parameters. 

\subsection{Non-staggered Formulation}
Let the finite element mesh, $\Omega^h=\{ \Omega_j \}$, with 
\[
{\Omega} = \bigcup_j {\Omega}_j,
\]
be a discretization of $\Omega$ consisting of geometry-conforming and non-overlapping  (possibly curved) elements with piecewise smooth boundaries.  

On each element ${\Omega}_j$, the approximation to the displacement, $u^h$, and the approximation to the velocity, $v^h$, are elements of some finite dimensional spaces $U^h$ and $V^h$ respectively. Then, the non-staggered energy-based discontinuous Galerkin method can be stated as follows. On each element ${\Omega}_j$, require that for all test functions 
\[
\phi \in U^h, \ \  \psi \in V^h,
\]
the following variational formulation holds: 
\begin{subequations}
\begin{eqnarray}
 \int_{{\Omega}_j} c^2 \nabla \phi \cdot \left( \f {\pa \nabla u^h}{\pa t} - \nabla v^h \right){dx}  &=& 
\int_{\pa {\Omega}_j} (c^2 \nabla \phi \cdot \vec{n}) 
\left( v^{\ast} -  v^h \right) ds,  \label{var1} \\
 \int_{{\Omega}_j}  \psi  \f {\pa v^h}{\pa t} + \nabla \psi \cdot (c^2\nabla u^h) - \psi f dx   &=&
\int_{\pa {\Omega}_j} \psi ( (c^2 \nabla u ) \cdot \vec{n})^{\ast}ds.  \label{var2}
\end{eqnarray}
\end{subequations}
As described in \cite{Upwind2} the energy is invariant to constants and this necessitates an additional equation complementing \eqref{var1}
\begin{equation}
 \int_{{\Omega}_j}  \left( \f {\pa  u^h}{\pa t} -  v^h\right) dx = 0,\quad \forall j.  \label{extra}
\end{equation}
Here $v^\ast$ and $( (c^2 \nabla u ) \cdot \vec{n})^{\ast}$ are numerical fluxes computed from the averages and jumps of function values and derivatives. Arbitrarily labeling values from adjacent elements $1$ and $2$ we recall the standard notation:
\begin{eqnarray}
\dsl v^h \dsr_{\alpha} & = & \f {1}{2} \left( \alpha v^{h,1} + (1- \alpha) v^{h,2} \right) \nonumber \\ \dbl v^h \dbr & = & v^{h,1} \vec{n}^1 + v^{h,2} \vec{n}^2 , \label{jav} \\
\dsl c^2 \nabla u^h \dsr_{\alpha} & = & \f {1}{2} \left( \alpha c^2 \nabla u^{h,1} + (1- \alpha) c^2 \nabla u^{h,2} \right), \nonumber \\ \dbl c^2 \nabla u^h \dbr & = & c^2 \nabla u^{h,1} \cdot \vec{n}^1 + c^2 \nabla u^{h,2} \cdot \vec{n}^2 . \label{jagu}
\end{eqnarray}
We then set
\begin{equation}
v^{\ast}=\dsl v^h \dsr_{\alpha} - \beta \dbl c^2 \nabla u^h \dbr , \label{vdiamond}    
\end{equation}
\begin{equation}
\left( c^2 \nabla u \cdot \vec{n} \right)^{\ast} = \dsl c^2 \nabla u^h \dsr_{1-\alpha}  \cdot \vec{n} - \tau \dbl v^h \dbr \cdot \vec{n} . \label{wstar} 
\end{equation} 
Here $\beta \geq 0$ is an upwinding parameter with units of $c^{-1}$ and $\tau \geq 0$ is an upwinding parameter with units of $c$. 
When $\beta=\tau=0$, one can recover the commonly used central fluxes by choosing $\alpha=1/2$, and alternating fluxes with $\alpha=0$ or $1$.

\subsection{Staggered Formulation}

We now consider two structured finite element meshes, $\Omega^h=\{ \Omega_j \}$ and $\Omega^{\diamond,h}=\{ \Omega_k^{\diamond} \}$ 
\[
\Omega = \bigcup_j\Omega_j  = \bigcup_k\Omega_k^\diamond.
\]
We assume each mesh consists of geometry-conforming and non-overlapping (possibly curved) quadrilaterals (or hexahedra) with piecewise smooth boundaries. We assume that the meshes are staggered. More precisely, away from non-periodic boundaries we assume that all quadrilaterals (hexahedra) are straight sided and convex and that all vertices have valence 4 (6). By staggering we mean that, away from boundaries, the vertices of the mesh $\Omega^{\diamond,h}$ coincide with the centers (defined as the vertex, side or area/volume centroid) of the elements in $\Omega^h$. 

For consistency with the theoretical and computational results to follow, we take the approximation to the velocity, $v^h$, restricted to an element ${\Omega}^\diamond_k$ in $\Omega^{\diamond,h}$, to be a tensor product polynomial in $(\mathbb{Q}^{q_v} ({\Omega}^\diamond_k))^d$ while the approximation to the displacement, $u^h$, restricted to an element ${\Omega}_j$ in $\Omega^h$, is taken to be a tensor product polynomial in  $(\mathbb{Q}^{q_u} ({\Omega}_j))^d$. Here $q_u\in\mathbb{N}$, $q_v\in\mathbb{N}\cup\{0\}$.

The staggered energy-based discontinuous Galerkin method then can be stated as follows. On each element $\Omega_j$ and $\Omega_k^\diamond$, require that for all test functions 
\[
\phi \in (\mathbb{Q}^{q_u} ({\Omega}_j) )^d, \ \  \psi \in (\mathbb{Q}^{q_v} ({\Omega}_k^\diamond))^d,
\]
the following variational formulation holds: 
\begin{subequations}
\begin{eqnarray}
	\int_{{\Omega}_j} \left( c^2 \nabla \phi \cdot \frac {\pa \nabla u^h}{\pa t} +  \nabla \cdot (c^2 \nabla \phi ) v^h \right) dx  &=&
	\int_{\pa {\Omega}_j} (c^2 \nabla \phi \cdot \vec{n}) 
	v^{\ast} ds,  \label{var1_s} \\
	\int_{{\Omega}_k^\diamond}  \psi  \frac {\pa v^h}{\pa t} + \nabla \psi \cdot (c^2 \nabla u^h) - \psi f dx  &=&
	\int_{\pa {\Omega}_k^\diamond} \psi ( (c^2 \nabla u ) \cdot {\bf n})^{\ast}ds.  \label{var2_s}
\end{eqnarray}
\end{subequations}
As with the non-staggered formulation we must complement \eqref{var1_s} with the equation
\begin{equation}
\int_{{\Omega}_j}   \left(\frac {\pa  u^h}{\pa t} -  v^h \right) dx = 0,\quad \forall j.  \label{extra_s}
\end{equation}
Again, here  $v^\ast$ and $( (c^2 \nabla u ) \cdot \vec{n})^{\ast}$ are numerical fluxes as in (\ref{vdiamond})-(\ref{wstar}). However, taking account of the staggering, we note that $v^h$ is single valued at $\pa \Omega_j$ and $c^2 \nabla u$ is single valued at $\pa \Omega_k^{\diamond}$ so the choice of $\alpha$ is not relevant. Lastly we note that the integrals of gradients in the variational form as well as in the calculations below are understood to be piecewise-defined in subdomains where the functions are smooth. For example, the integral in (\ref{var2_s}) includes boundaries of elements in $\Omega^h$ across which $u^h$ is discontinuous. We do not, here, interpret $\nabla u^h$ in a distributional sense and so no additional boundary terms are implied. 

We note the difference between (\ref{var1}) and (\ref{var1_s}). If
the term $\nabla \cdot (c^2 \nabla \phi) v^h$ is integrated by parts in (\ref{var1_s}), terms involving the jump in $v^h$ across boundaries of dual mesh elements will appear. These play a role in the energy estimate we now derive. For the subsequent analysis we set $f=0$ for simplicity as the source function plays no role in determining time step stability constraints.

Define the discrete energy to be
\begin{equation}
\mathcal{E}^h (t) = \f {1}{2} \sum_k \int_{\Omega_k^{\diamond}} \left( v^h \right)^2\ dx + \f {1}{2} \sum_j \int_{\Omega_j} c^2 \arrowvert \nabla u^h \arrowvert^2\ dx . \label{energyh} 
\end{equation}
To start, we assume periodic boundary conditions. 
Choosing $\phi = u^h$ in \eqref{var1_s}, integrating by parts, and using the fact that $v^h$ is single valued on $\pa \Omega_j$ we find
\begin{eqnarray*}
\f {1}{2} \f {d}{dt} \sum_j \int_{\Omega_j} c^2 \arrowvert \nabla u^h \arrowvert^2\ dx & = & \sum_j \int_{\Omega_j} c^2 \nabla u^h \cdot \nabla v^h \ dx - \sum_k \int_{\pa \Omega_k^\diamond} c^2 \nabla u^h \cdot \dbl v^h \dbr \ ds \\ & & - \beta \sum_j \int_{\pa \Omega_j} \dbl c^2 \nabla u^h \dbr^2 \ ds, 
\end{eqnarray*}
Similarly we find
\begin{eqnarray*}
\f {1}{2} \f {d}{dt} \sum_k \int_{\Omega_k^{\diamond}} \left( v^h \right)^2\ dx & = & - \sum_k \int_{\Omega_k^{\diamond}} c^2 \nabla u^h \cdot \nabla v^h \ dx + \sum_k \int_{\pa \Omega_k^{\diamond}} c^2 \nabla u^h \cdot \dbl v^h \dbr \ ds\\ & & - \tau \sum_k \int_{\pa \Omega_k^{\diamond}} \arrowvert \dbl v^h \dbr \arrowvert^2 \ ds.
\end{eqnarray*}
Summing these equations, we see that the left-hand side is the time derivative of the discrete energy. Since $\bigcup_j \Omega_j = \bigcup_k \Omega_k^{\diamond}$ and recalling the piecewise definition of the integrals we conclude that the terms involving $c^2 \nabla u^h \cdot \nabla v^h$ cancel. Thus we conclude:
\begin{equation}
\f {d \mathcal{E}^h}{dt} = - \beta \sum_j \int_{\pa \Omega_j} \dbl c^2 \nabla u^h \dbr^2 \ ds
- \tau \sum_k \int_{\pa \Omega_k^{\diamond}} \arrowvert \dbl v^h \dbr \arrowvert^2 \ ds. \label{Energyest}
\end{equation}

At nonperiodic boundaries we alter the staggered mesh so that elements from both $\Omega^h$ and $\Omega^{\diamond,h}$ conform to $\pa \Omega$. Now the imposition of the boundary conditions is the same as for the non-staggered formulation. For example, recalling (\ref{eq:bc}) we may set
\begin{eqnarray}
v^{\ast}  & = & \f {\kappa}{\gamma + \kappa} \left( v^h - c \nabla u^h \cdot \vec{n} \right), \label{vdiamondbc} \\
\left( c^2 \nabla u^h \cdot \vec{n} \right)^{\ast} & = & \f {\gamma}{\gamma + \kappa} \left( c^2 \nabla u^h \cdot \vec{n} - c v^h \right) . \label{wstarbc} 
\end{eqnarray}
Then the contribution of the nonperiodic boundaries to the energy derivative can be shown to be nonpositive. The mesh modification at these boundaries will preclude taking global large time steps. To maintain the efficiency of the staggered scheme we will then use local time stepping in the vicinity of the boundaries; see section \ref{sec:timestepping} for details.

\section{Operator Bounds on Periodic Domains in One Space Dimension}\label{sec:operatorbounds}

In this section we use the techniques from \cite{TAMECFL,fli} to establish bounds for the energy-based DG and the staggered energy-based DG spatial operator for the second-order wave equation (\ref{wave_1d_u}) and (\ref{wave_1d_v}) in one space dimension. This allows us to invoke the Kreiss-Wu theory  \cite{kreiss1993stability} combined with the energy estimates to establish the stability of fully discrete locally stable explicit time-stepping schemes. 

We restrict the analysis to uniform grids, periodic boundary conditions, and constant coefficients. As the key ingredient to taming the CFL condition is to evaluate certain terms with derivatives only  \emph{near} the element centers,  we expect that the analysis can be extended to smoothly varying grids and to variable coefficients. The numerical experiments demonstrate the efficiency of the method for a variable coefficient problem. It may also be possible to extend the analysis to problems with Dirichlet or Neumann boundary conditions by using the image principle, however we don't pursue this here.         

\subsection{Operator Bounds for the Non-Staggered Formulation}

Now, consider the one dimensional wave equation in a uniform medium  
\begin{eqnarray}
u_t &=& v,\label{wave_1d_u} \\
v_t &=& c^2 u_{xx}\label{wave_1d_v},
\end{eqnarray}
on the domain $x\in[x_{\rm L}, x_{\rm R}] \equiv \Omega$. Let the domain be discretized by a grid \mbox{$x_j = x_{\rm L} + j h,$} $j = 0,\ldots,N,\, h = (x_{\rm R}-x_{\rm L}) / N$, and $I_j=[x_j, x_{j+1}]$. Associated with the grid, we define two  broken finite element spaces
\[
U_h^{q_u} = \{ w : w|_{I_j} \in \mathbb{Q}^{q_u}(I_j)\;\forall j\}, \ \ \ \ V_h^{q_v} = \{ w : w|_{I_j} \in \mathbb{Q}^{q_v}(I_j)\;\forall j\}.
\]
Here and below $\mathbb{Q}^{q_u}(I_j)$ is the space of polynomials of degree up to $q_u$ in $I_j$, $q_u\in\mathbb{N}$, and $q_v\in\mathbb{N}\cup\{0\}$. In addition, we denote $\phi^\pm(x) = \lim_{\varepsilon \rightarrow 0\pm} \phi(x+ \varepsilon)$.

The energy-based DG scheme then consists of finding $u^h(\cdot,t) \in U_h^{q_u}$ and $v^h(\cdot,t) \in V_h^{q_v}$ such that for any $\phi \in U_h^{q_u}$ and $\psi \in V_h^{q_v}$ and for all $j$
\begin{subequations}
\label{eq:method_ns}
\begin{eqnarray}
\int_{x_j}^{x_{j+1}} c^2 \phi_x\left(\frac{\partial u^h_x}{\partial t}-v^h_x\right)dx &=& c^2 \phi_x^-(v^\ast-v^{h,-})\Big|_{x_{j+1}}- c^2 \phi_x^+(v^\ast-v^{h,+})\Big|_{x_j},\label{eq:method_ns1} \\
\int_{x_j}^{x_{j+1}}\psi\frac{\partial v^h}{\partial t} + c^2 \psi_x u^h_x dx &=& c^2 \psi^- u_x^\ast \Big|_{x_{j+1}}- c^2 \psi^+ u_x^\ast \Big|_{x_j},\label{eq:method_ns2}\\
\int_{x_j}^{x_{j+1}} \frac{\partial u^h}{\partial t}-v^h dx&=&0.\label{eq:method_ns3}
\end{eqnarray}
\end{subequations}

Assuming periodic boundary conditions we may add up the equations \eqref{eq:method_ns1}-\eqref{eq:method_ns2} in $j$ to find
\begin{align}
\int_{\Omega} c^2 \phi_x\frac{\partial u^h_x}{\partial t} + \psi\frac{\partial v^h}{\partial t} dx =\int_\Omega
c^2 \phi_x v^h_x - c^2 \psi_x u^h_x dx \hspace{1.0in} \label{eq:var_energy1} \\ 
+ \sum_j c^2 \left(
\phi_x^-(v^\ast-v^{h,-})\Big|_{x_{j+1}}- \phi_x^+(v^\ast-v^{h,+})\Big|_{x_j}
+\psi^- u_x^\ast \Big|_{x_{j+1}}-\psi^+ u_x^\ast \Big|_{x_j}\right).
\notag
\end{align}
Throughout, the spatial derivative of functions in any broken finite element space shall be understood  as being defined element by element.
To connect the element solutions in a stable fashion we use the numerical fluxes defined in (\ref{vdiamond})-(\ref{wstar}) and introduce the notation
 \begin{subequations}
\label{eq:sflux:L}
\begin{eqnarray}
v^\ast &=& \mathcal{H} ({v}^h, u_x^h)=   
\{\{v^h\}\}_\alpha  - \beta c^2 (u_{x}^{h,-} - u_{x}^{h,+}), \label{eq:sflux1:L} \\
u_{x}^\ast &=& \mathcal{G}(u_x^h, v^h)= \{\{u_{x}^{h}\}\}_{1-\alpha}  - \frac {\tau}{c^2} (v^{h,-}-v^{h,+}). \label{eq:sflux2:L}
\end{eqnarray}
\end{subequations}
Then the energy estimate (\ref{Energyest}) holds. We note that it can also be used to establish error estimates for different choices of $\alpha$, $\beta$ and $\tau$; see \cite{Upwind2} for details. 

We now establish bounds on the spatial operators which constrain the allowable time step sizes for explicit marching schemes. In particular we are interested in the dependence of these bounds on the approximation orders, $q_u$ and $q_v$ and will follow a similar analysis as in \cite{TAMECFL,fli}. With the choice of the numerical fluxes in (\ref{vdiamond})-(\ref{wstar}), an important observation is that the first two equations in \eqref{eq:method_ns} are coupled with \eqref{eq:method_ns3} in a one-way manner. That is, \eqref{eq:method_ns1}-\eqref{eq:method_ns2} will uniquely determine $w^h=u_x^h\in U_h^{q_u-1}$ and $v^h\in V_h^{q_v}$. Once $w^h$, $v^h$ are available, one can further recover the missing constant in $u^h$ on $[x_j, x_{j+1}]$ (i.e. in the form of the cell average of $u^h$) through \eqref{eq:method_ns3} for all $j$. As this last step is simply an integration in time it can not affect the numerical stability; see also the discussion in section \ref{sec:timestepping}. 

These considerations motivate us to define the operator
${\mathcal L}: U_h^{q_u-1}\times V_h^{q_v}\mapsto  U_h^{q_u-1}\times V_h^{q_v}$,  
\begin{eqnarray}
\int_{\Omega}\mathcal{L}(cw,v)(c\varphi,\psi)dx =&& c^2 \int_\Omega\varphi v_x - \psi_x w dx+ c^2 \sum_j \left(\psi^- \mG(w,v) \Big|_{x_{j+1}}-\psi^+ \mG(w,v) \Big|_{x_j}\right)\notag\\
&&+ c^2 \sum_j \left(
\varphi^-(\mH(v,w)-v^{-})\Big|_{x_{j+1}}- \varphi^+(\mH(v,w)-v^{+})\Big|_{x_j}\right) 
\label{eq:operator_ns}
\end{eqnarray}
for any $\varphi\in U_h^{q_u-1}$ and $ \psi\in V_h^{q_v},$
with the operator norm as
\begin{equation}
\label{eq:op:norm}
\left\|  \mathcal{L} \right\| \equiv 
\sup_{\substack{w,\varphi \in U_h^{q_u-1},\, v, \psi \in V_h^{q_v}\\
(w,v)\ne (0,0), (\varphi, \psi)\ne (0,0)}}
\frac{\int_{\Omega}\mathcal{L}(cw,v)(c\varphi,\psi)dx}{(\| cw \|_{L^2(\Omega)}^2+\| v \|_{L^2(\Omega)}^2)^{1/2}(\| c\varphi \|_{L^2(\Omega)}^2 + \| \psi \|_{L^2(\Omega)}^2)^{1/2}}.
\end{equation}

Once the bound is established for $\left\|\mL \right\|$, time step condition, $\Delta t \left\|  \mathcal{L} \right\| \le \mathcal{R}$, for method of lines discretization combined with locally stable one-step temporal methods will follow from  Kreiss-Wu theory  \cite{kreiss1993stability}.
Here $\mathcal{R}$ is defined as the radius of the largest semidisk in the closed left half complex plane contained in the stability domain of the method. 
It is well known, \cite{ketcheson2015absolute}, that one-step methods based on Taylor expansion with $q_{\rm T} = 3,4,7,8,11,12,15,16,\ldots$ terms are locally stable. For $q_{\rm T}$ of moderate size they have stability domains which grow with order. Thus, if we can establish a bound on $\left\|  \mathcal{L} \right\|$ that grows {\em linearly} in $q_u$ and $q_v$, we should expect that the fully discrete method can time-march at a CFL condition of $\mathcal{O}(1)$ when the spatial and temporal orders are matched. As the order increases this does not hold, but, as discussed in section \ref{sec:timestepping}, with the introduction of 
additional stages the size of the stability domain can be greatly increased. In what follows we will see that such a bound can be established for the staggered method with suitably chosen numerical fluxes but not for the non-staggered method. For the non-staggered method,  the bound on $\left\|  \mathcal{L} \right\|$ is quadratic in $q_u$ and $q_v$, and this will be proved next. We note that the quadratic dependence on the degrees is sharp as demonstrated numerically in \cite{Upwind2}. 

\begin{truth}
Let the energy-based DG spatial operator $\mathcal{L}$ be defined as in (\ref{eq:operator_ns}) with periodic boundary conditions and with numerical fluxes defined by
(\ref{vdiamond})-(\ref{wstar}).
 Then the following estimate holds:
 \begin{eqnarray}
 \label{bounds_dg}
\|\mL\| & \leq & C_1 \frac{c}{h}\max\{(q_u-1)^2, q_v^2\} \\
&  & + C_2 \frac{c}{h} \left(2\max\left\{c \beta q_u^2,\frac {\tau}{c} (q_v+1)^2\right\}+(|\alpha|+|1-\alpha|)(q_v+1)q_u\right) . \nonumber
\end{eqnarray}
Here $C_1\le 2 \sqrt{3}$ and $C_2 \le 2$
are two universal positive constants, independent of $q_u$, $q_v$, $h$, and $\alpha,\beta, \tau$.
\end{truth}
\begin{proof}  Consider any $w,\varphi \in U_h^{q_u-1}$ and $v,\psi\in V_h^{q_v}$. Applying element-wise integration by parts and the triangle inequality, we have 
\begin{eqnarray}
\int_{\Omega}\mathcal{L}(cw,v)(c\varphi,\psi)dx =&& -c^2 \int_\Omega\varphi_{x} v +\psi_x w dx
+ c^2 \sum_j \left(
\varphi^-\mH(v,w)\Big|_{x_{j+1}}- \varphi^+\mH(v,w)\Big|_{x_j}\right)
\notag\\
&&+ c^2 \sum_j \left(\psi^- \mG(w,v) \Big|_{x_{j+1}}-\psi^+ \mG(w,v) \Big|_{x_j}\right)
\leq \Lambda_1+\Lambda_2+\Lambda_3,
\end{eqnarray}
with
\begin{eqnarray*}
\Lambda_1 &=& c^2 \sum_j\int_{x_j}^{x_{j+1}} \left|\varphi_{x} v \right|dx+\int_{x_j}^{x_{j+1}}\left|\psi_x w\right| dx,\\
\Lambda_2 &=& c^2 
\sum_j \left|\varphi^-\mH(v,w)\right|\Big|_{x_{j+1}}+ \left|\varphi^+\mH(v,w)\right|\Big|_{x_j},\\
\Lambda_3 &=& c^2 
\sum_j \left|\psi^- \mG(w,v)\right| \Big|_{x_{j+1}}+\left|\psi^+ \mG(w,v)\right| \Big|_{x_j}.
\end{eqnarray*}

We now bound each of the terms, starting with the volume term $\Lambda_1$. By applying the Cauchy-Schwarz inequality,
we have
\begin{equation*}
\Lambda_1 \leq c \sum_j\|c\varphi_{x}\|_{L^2(I_j)}\|v\|_{L^2(I_j)} + \|\psi_x\|_{L^2(I_j)}\| cw\|_{L^2(I_j)}.
\end{equation*}
For $\Lambda_2$ and $\Lambda_3$, we use the definitions of $\mH(v, w)$ and $\mG(w, v)$ as well as the triangle inequality, and arrive at
\begin{align*}
\Lambda_2 \leq c \sum_j& \left|c \varphi^-\right|\left(|\alpha|\left|v^{-}\right| + |1-\alpha| \left|v^{+}\right| + c \beta \left|c w^{-}\right| + c \beta \left| c w^{+}\right|\right)\Big|_{x_{j+1}}\\
+&\left|c \varphi^+\right|\left(|\alpha|\left|v^{-}\right| + |1-\alpha| \left|v^{+}\right|+  c \beta \left|c w^{-}\right| + c \beta \left| cw^{+}\right|\right)\Big|_{x_j},\\
\Lambda_3 \leq c \sum_j &\left|\psi^-\right|\left(|1-\alpha|\left|c w^{-}\right|+|\alpha| \left|c w^{+}\right|+ \frac {\tau}{c} \left|v^{-}\right|+\frac{\tau}{c}\left|v^{+}\right|\right)\Big|_{x_{j+1}} \\
+ &\left|\psi^+\right|\left(|1-\alpha|\left|c w^{-}\right|+|\alpha|\left| c w^{+}\right|+ \frac{\tau}{c} \left|v^{-}\right|+\frac{\tau}{c}\left|v^{+}\right|\right)\Big|_{x_j}.
\end{align*}

Now we recall some standard inverse inequalities for polynomials spaces \cite{fli}; there exist positive constants $\hat{C}_1\leq \sqrt{3}$, $\hat{C}_2\leq\frac{\sqrt{2}}{2}$, such that $\forall p\in \mathbb{Q}^k([-1,1])$,
\begin{equation}
\|p_x\|_{L^2([-1,1])}\leq \hat{C}_1 k^2\|p\|_{L^2([-1,1])}, \quad p(x) \leq \hat{C}_2 (k+1)\|p\|_{L^2([-1,1])}\; \forall x\in [-1,1].
\label{eq:inv1}
\end{equation}
By applying these inverse inequalities, with a linear scaling from $[-1, 1]$ to $I_j$, and Cauchy-Schwarz inequality, we find that
\begin{align*}
\Lambda_1 &\leq 2 \hat{C}_1 \frac {c}{h} \sum_j \left( (q_u-1)^2 \|c \varphi\|_{L^2(I_j)}\|v\|_{L^2(I_j)} +q_v^2 \|\psi\|_{L^2(I_j)}\| cw \|_{L^2(I_j)}\right)\\
&\leq C_1 \frac {c}{h}\max\{(q_u-1)^2, q_v^2\}\left(\|cw\|^2_{L^2(\Omega)}+\|v\|^2_{L^2(\Omega)}\right)^{1/2}  
\!\left(\|c\varphi\|_{L^2(\Omega)}^2+\|\psi\|_{L^2(\Omega)}^2\right)^{1/2},
\end{align*}
and similarly, using a linear scaling from $[-1, 1]$ to $I_s$ (with $s=j,j\pm 1$), we have
\begin{align*}
\Lambda_2\leq& 2 \hat{C}_2^2 \frac {c}{h} \sum_j \Big( q_u(q_v+1)  \|c \varphi\|_{L^2(I_j)}\left(|\alpha|\|v\|_{L^2(I_{j-1}\cup I_{j})}+|1-\alpha|\|v\|_{L^2(I_j\cup I_{j+1})}\right)\\
&\;\;\;\;\;\; + c \beta q_u^2 \|c \varphi\|_{L^2(I_j)}\left(2\|cw\|_{L^2(I_j)}+\|cw\|_{L^2(I_{j-1}\cup I_{j+1})}\right)\Big),
\end{align*}
\begin{align*}
\Lambda_3\leq 
&2 \hat{C}_2^2 \frac {c}{h} \sum_j \Big( q_u(q_v+1)\|\psi\|_{L^2( I_j)}\left(|1-\alpha|\|cw\|_{L^2(I_{j-1}\cup I_j)}+|\alpha|\|cw\|_{L^2(I_{j}\cup I_{j+1})}\right)\\
&\;\;\; \;\;\;+\frac{\tau (q_v+1)^2}{c}\|\psi\|_{L^2(I_j)}\left(2\|v\|_{L^2(I_j)}+\|v\|_{L^2(I_{j-1}\cup I_{j+1})}\right)\Big),
\end{align*}
and hence

\begin{align*}
\Lambda_2+\Lambda_3\leq& C_2 \frac{c}{h} \left(2\max\left\{c \beta q_u^2, \frac {\tau}{c} (q_v+1)^2\right\}+(|\alpha|+ |1-\alpha|)(q_v+1)q_u\right)\\
&\;\;\;\;
\left(\|cw\|^2_{L^2(\Omega)}+\|v\|^2_{L^2(\Omega)}\right)^{1/2}  
\left(\|c \varphi\|_{L^2(\Omega)}^2+\|\psi\|_{L^2(\Omega)}^2\right)^{1/2}.
\end{align*}

Finally by adding up $\Lambda_j, j=1, 2, 3$, and based on  the operator norm $\|\mL\|$ in \eqref{eq:op:norm}, we reach the bound in \eqref{bounds_dg}. 
\end{proof}

\subsection{Operator Bounds for the Staggered Formulation}\label{staggered_formula_1d}
For the staggered version of the method we introduce element centers $\rho_j = x_{j+\frac{1}{2}}
 = (x_j+x_{j+1})/2$ as well as the staggered grid composed of the elements $I_{j+\frac{1}{2}} =
 [\rho_{j},\rho_{j+1}] 
 \; \forall j$. Associated with both grids, we define two broken finite element spaces 
\[
U_h^{q_u} = \{ w : w|_{I_j} \in \mathbb{Q}^{q_u}(I_j)\; \forall j\}, \ \ \ \ \widetilde{V}_h^{q_v} = \{ w : w|_{I_{j+\frac{1}{2}}} \in \mathbb{Q}^{q_v}(I_{j+\frac{1}{2}}) \; \forall j\}.
\]

The staggered energy-based DG scheme then consists of finding $u^h(\cdot,t) \in U_h^{q_u}$ and $\tilde{v}^h(\cdot,t) \in \widetilde{V}_h^{q_v}$ such that for any $\phi \in U_h^{q_u}$ and $\psi \in \widetilde{V}_h^{q_v}$ and for all $j$ 
\begin{subequations}
\label{eq:staggered} 
\begin{align}
\int_{x_j}^{x_{j+1}} c^2 \phi_x \frac{\partial u_x^h}{\partial t}dx
+\underbrace{\int_{x_j}^{\rho_j} c^2 \phi_{xx} \tilde{v}^{h} dx 
+ \int_{\rho_j}^{x_{j+1}} c^2 \phi_{xx} \tilde{v}^h dx}_{\int_{x_j}^{x_{j+1}}c^2\phi_{xx} \tilde{v}^{h} dx } 
 \notag \\ 
=c^2 \phi^-_x v^\ast\big|_{x_{j+1}}- c^2 \phi^+_x v^\ast\big|_{x_j}, 
\hspace{1.5in} \label{eq:staggered1} \\
\int_{\rho_j}^{\rho_{j+1}} \psi \frac{\partial \tilde{v}^{h}}{\partial t}dx
+\underbrace{\int_{\rho_j}^{x_{j+1}} c^2 \psi_x u_x^h dx
 + \int_{x_{j+1}}^{\rho_{j+1}} c^2 \psi_x  u_x^{h}dx}_{\int_{\rho_j}^{\rho_{j+1}} c^2\psi_x  u_x^{h}dx} 
  \notag \\ 
=c^2 \psi^- u_x^\ast\big|_{\rho_{j+1}}- c^2 \psi^+ u_x^\ast\big|_{\rho_j}, 
\hspace{1.5in} \label{eq:staggered2}\\
 \int_{x_j}^{x_{j+1}} \left(\frac{\partial u^h}{\partial t}-\tilde{v}^h\right)dx=0.
\hspace{1.5in} \label{eq:staggered3}
\end{align}
\end{subequations}
Note that the second and third integrals in (\ref{eq:staggered1}) (resp. in (\ref{eq:staggered2})) are against $\tilde{v}^h$ (resp. $u^h$) from two elements.

Explicitly we write the flux terms 
\begin{subequations}
\label{eq:sflux}
\begin{eqnarray}
v^\ast &=& H(\tilde{v}^h, u_x^h)=   
\tilde{v}^{h}  - \beta c^2 (u_{x}^{h,-} - u_{x}^{h,+}), \label{eq:sflux1} \\
u_{x}^\ast &=& G(u_x^h, \tilde{v}^h)= u_{x}^{h} - \frac {\tau}{c^2} (\tilde{v}^{h,-}-\tilde{v}^{h,+}) , \label{eq:sflux2}
\end{eqnarray}
\end{subequations}
noting that there is no ambiguity for $\tilde{v}^h$ in \eqref{eq:sflux1} and $u_x^h$ in \eqref{eq:sflux2} since they are evaluated at the element centers and are uniquely defined.

Assuming periodic boundary conditions, we apply integration by parts to  \eqref{eq:staggered1} and \eqref{eq:staggered2}, add them up in $j$ and reach  the equality
\begin{align}
\label{sum_scheme}
&\int_{\Omega} c^2 \phi_x\frac{\partial u^h_x}{\partial t} + \psi \frac{\partial\tilde{v}^h}{\partial t}dx = 
c^2 \sum_j
\int_{x_j}^{x_{j+1}} \left(\phi_{x} \tilde{v}^h_x 
-\psi_x u_x^h\right)dx\\
&+ c^2 \sum_j  \left(\phi_x^-(v^\ast-\tilde{v}^h)\big|_{x_{j+1}}-
\phi_x^+(v^\ast-\tilde{v}^h)\big|_{x_{j}} 
+\phi_x(\tilde{v}^{h,+}-\tilde{v}^{h,-}) \big|_{\rho_j}\right)\notag\\
&
+ c^2 \sum_j\left(\psi^- u_x^\ast\big|_{\rho_{j+1}}-
\psi^+ u_x^\ast\big|_{\rho_j}\right) , \notag
\end{align}
with semi-discrete stability of the method following directly from (\ref{Energyest}).

As for the non-staggered scheme, the first two equations \eqref{eq:staggered1}-\eqref{eq:staggered2} will determine $w^h=u^h_x\in U_h^{q_u-1}$ and $\tilde{v}^h\in \widetilde{V}_h^{q_v}$, and hence 
we define the operator
${\mL}_c: U_h^{q_u-1}\times \widetilde{V}_h^{q_v}\mapsto  U_h^{q_u-1}\times \widetilde{V}_h^{q_v}$, satisfying
\begin{align}\label{spatial_operator_s}
\int_{\Omega}\mathcal{L}_c(cw,v)(c\varphi,\psi) dx=&- c^2  
\sum_j\left(\int_{x_j}
^{x_{j+1}} \varphi_{x} v dx -\varphi^-H(v, w)|_{x_{j+1}}+
\varphi^+H(v, w)|_{x_j}\right) \notag\\
&- c^2 \sum_j\left(\int_{\rho_j}
^{\rho_{j+1}} \psi_x  wdx 
-\psi^- G(w, v)|_{\rho_{j+1}}+ \psi^+ G(w, v)|_{\rho_j}\right)
\end{align}
for any $\varphi\in U_h^{q_u-1}$ and $ \psi\in\widetilde{V}_h^{q_v},$ with the operator norm as
\begin{equation*}
\left\|  \mathcal{L}_c \right\| \equiv 
\sup_{\substack{w,\varphi \in U_h^{q_u-1},\, v, \psi \in \widetilde{V}_h^{q_v}\\
(w,v)\ne (0,0), (\varphi, \psi)\ne (0,0)}} \frac{\int_{\Omega}\mathcal{L}_c(cw,v)(c\varphi,\psi)dx}{\left(\|cw\|_{L^2(\Omega)}^2+\|v\|_{L^2(\Omega)}^2\right)^{1/2}\left(\|c\varphi\|_{L^2(\Omega)}^2+\|\psi\|_{L^2(\Omega)}^2\right)^{1/2}}.
\end{equation*}

The theorem governing the bound on the operator $\mL_c$ has a similar form as the non-staggered case, namely, with the quadratic dependence on $q_u$ and $q_v$.  But when the upwinding parameters $\beta$ and $\tau$ are set to zero and the numerical fluxes become purely central, or when $\beta$ and $\tau$ are chosen to be order-dependent, the result is significantly stronger. We now state and prove this theorem. 
\begin{truth}\label{thm:staggerbound}
Let the staggered energy-based DG spatial operator $\mathcal{L}_c$ be defined as in (\ref{spatial_operator_s}), with numerical fluxes defined by \eqref{eq:sflux} and  periodic boundary conditions. Then the following estimate holds: 
\begin{equation}\label{bounds_cdg} \|\mathcal{L}_c\| \leq
    \frac{c}{h}C_{3,\frac{1}{2}} (q_u+q_v-1)
   +\frac{c}{h}C_{4,\frac{1}{2}} \sqrt{q_u(q_v+1)}+\frac{c}{h}C_5 \max\left\{c\beta   q_u^2, \frac {\tau}{c} (q_v+1)^2\right\}.
   \end{equation}
   In particular, when $\beta  =\tau= 0$, or
   when $\beta  = \frac {\hat{\beta}}{q_u c}$, $\tau=\frac {c \hat{\tau}}{q_v+1}$ with  fixed dimensionless constants $\hat{\beta}$, $\hat{\tau}$, the result is strengthened in that  $\|\mathcal{L}_c\|$ is bounded linearly in  $q_u$ and  $q_v$.
Here  $C_{3,\frac{1}{2}}\leq \frac{8\sqrt{3}+4}{3}$, $C_{4,\frac{1}{2}}\leq\frac{128}{\sqrt{3}\pi}$, and $C_5 \leq 4$  are universal positive  constants, independent of $q_u$, $q_v$, $h$, and $\beta$  , $\tau$.
\end{truth}
\begin{proof}

By  partitioning $I_j$ into $(x_j, x_{j+\frac{1}{4}})$, $(x_{j+\frac{1}{4}}, x_{j+\frac{3}{4}})$, $(x_{j+\frac{3}{4}}, x_{j+1})$, partitioning $I_{j+\frac{1}{2}}$ into $(x_{j+\frac{1}{2}},x_{j+\frac{3}{4}})$, $(x_{j+\frac{3}{4}},x_{j+\frac{5}{4}})$, $(x_{j+\frac{5}{4}},x_{j+\frac{3}{2}})$, and performing integration by parts on those sub-intervals of length $h/4$, we find  for any $w,\varphi \in U^{q_u-1}_h$ and $v,\psi\in \tilde{V}^{q_v}_h$,
\begin{align*}
   & \int_{\Omega}\mathcal{L}_c(cw,v)(c\varphi,\psi) dx\\
    =&
c^2    \sum_j\left( \int_{x_j}^{x_{j+\frac{1}{4}}}\varphi v_xdx-\int_{x_{j+\frac{1}{4}}}^{x_{j+\frac{3}{4}}}\varphi_{x}vdx+\int_{x_{j+\frac{3}{4}}}^{x_{j+1}}\varphi v_xdx\right) \\
  &+ c^2 \sum_j\left(\int_{x_{j+\frac{1}{2}}}^{x_{j+\frac{3}{4}}}\psi w_{x}dx - \int_{x_{j+\frac{3}{4}}}^{x_{j+\frac{5}{4}}}\psi_x wdx+\int_{x_{j+\frac{5}{4}}}^{x_{j+\frac{3}{2}}}\psi w_{x}dx\right)\\
  &+ c^2 \sum_j\left( -\beta c^2 \varphi^+\left(w^+-w^-\right)\big|_{x_j} + \beta c^2 \varphi^-\left(w^+-w^-\right)\big|_{x_{j+1}} \right)\\
&+ c^2 \sum_j\left(-\frac{\tau}{c^2}\psi^{+}\left(v^{+}-v^{-}\right)\big|_{x_{j+\frac{1}{2}}}+\frac{\tau}{c^2}\psi^{-}\left(v^{+}-v^{-}\right)\big|_{x_{j+\frac{3}{2}}}\right)\\
&+c^2\sum_j\left(-\psi w\big|_{x_{j+\frac{3}{4}}}+\psi w\big|_{x_{j+\frac{5}{4}}}-\varphi v\big|_{x_{j+\frac{1}{4}}} + \varphi v\big|_{x_{j+\frac{3}{4}}}\right).
\end{align*}

Then by the triangle inequality, we have 
\begin{equation} \int_{\Omega}\mathcal{L}_c(cw,v)(c\varphi,\psi) dx\leq \sum_{k=1}^3 \Theta_{u,k}+ \sum_{k=1}^3 \Theta_{v,k},
\label{eq:Tuv0}
\end{equation}
where
\begin{eqnarray*}
    \Theta_{u,1} 
   &=& c \sum_j\left( \int_{x_{j+\frac{3}{4}}}^{x_{j+\frac{5}{4}}}\big|c \varphi v_x\big|dx+\int_{x_{j+\frac{1}{4}}}^{x_{j+\frac{3}{4}}}\big| c \varphi_{x}v\big|dx\right), \\
    \Theta_{u,2} &= & c \sum_j\left( \big|c \varphi \big|\big|v\big|\Big|_{x_{j+\frac{1}{4}}} + \big| c \varphi\big|\big|v\big|\Big|_{x_{j+\frac{3}{4}}}\right),\\
    \Theta_{u,3} &=& \beta c^2  \sum_j\left( \big|c \varphi^+\big|\big| c w^+\big|\Big|_{x_j}+\big| c \varphi^+\big|\big| c w^{-}\big|\Big|_{x_j} + \big| c \varphi^{-}\big|\big| c w^{+}\big|\Big|_{x_{j+1}}+\big| c \varphi^-\big|\big| c w^-\big|\Big|_{x_{j+1}}\right),\\
 \Theta_{v,1} 
&= & c \sum_j\left(\int_{x_{j+\frac{1}{4}}}^{x_{j+\frac{3}{4}}}\big|\psi cw_{x}\big|dx + \int_{x_{j+\frac{3}{4}}}^{x_{j+\frac{5}{4}}}\big|\psi_x cw\big|dx\right), \\
\Theta_{v,2} &= & c \sum_j\left(\big|\psi\big|\big| c w\big|\Big|_{x_{j+\frac{3}{4}}}+\big|\psi\big|\big| cw \big|\Big|_{x_{j+\frac{5}{4}}}\right),\\
\Theta_{v,3} &=& \tau \sum_j\left(\big|\psi^{+}\big|\big|v^{+}\big|\Big|_{x_{j+\frac{1}{2}}}+\big|\psi^{+}\big|\big|v^{-}\big|\Big|_{x_{j+\frac{1}{2}}} +\big|\psi^{-}\big|\big|v^{+}\big|\Big|_{x_{j+\frac{3}{2}}}+\big|\psi^{-}\big|\big|v^{-}\big|\Big|_{x_{j+\frac{3}{2}}}\right).
\end{eqnarray*}

In preparation,  we  recall some inverse inequalities for polynomials spaces (e.g. see Lemmas 3-4 in \cite{fli}): there exist positive constants $\hat{C}_{3,\frac{1}{2}}\leq \frac{4\sqrt{3}+2}{3}$, $\hat{C}_{4,\frac{1}{2}}\leq 4\sqrt{\frac{3}{\pi}}(\frac{3}{4})^{-1/4}=\sqrt{\frac{32}{\sqrt{3}\pi}}$, such that $\forall p\in \mathbb{Q}^k([-1,1])$,
\begin{equation}
\|p_x\|_{L^2([-\frac{1}{2},\frac{1}{2}])}\leq \hat{C}_{3,\frac{1}{2}} k\|p\|_{L^2([-1,1])}, \quad p(\pm\frac{1}{2}) \leq \hat{C}_{4,\frac{1}{2}} \sqrt{k+1}\|p\|_{L^2([-1,1])}.
\label{eq:inv2}
\end{equation}
These inverse inequalities display different dependence on polynomial degree $k$ from those in  \eqref{eq:inv1}, and they will play a key role for our estimate with the desired dependence on the approximation order $q_u$, $q_v$. 

We start with bounding the volume integral terms $\Theta_{u,1}$,
$\Theta_{v,1}$.
By using Cauchy-Schwarz inequality, and the first inverse inequality in \eqref{eq:inv2} with a linear scaling from $[-1, 1]$ to $I_j$ (or to $I_{j-\frac{1}{2}}=[x_{j-\frac{1}{2}},x_{j+\frac{1}{2}}]$), we get
\begin{eqnarray*}
\Theta_{u,1} & \leq & c \sum_j
\| c \varphi\|_{L^2([x_{j+\frac{3}{4}},x_{j+\frac{5}{4}}])}\|v_x\|_{L^2([x_{j+\frac{3}{4}},x_{j+\frac{5}{4}}])} + \| c \varphi_{x}\|_{L^2([x_{j+\frac{1}{4}},x_{j+\frac{3}{4}}])}\|v\|_{L^2([x_{j+\frac{1}{4}},x_{j+\frac{3}{4}}])}\\
& \leq &  
\frac{c}{h} 2 \hat{C}_{3,\frac{1}{2}} \sum_j \left(q_v\| c \varphi\|_{L^2(I_{j+\frac{1}{2}})}\|v\|_{L^2(I_{j+\frac{1}{2}})}+(q_u-1)\| c \varphi\|_{L^2(I_j)}\|v\|_{L^2(I_j)}\right)\\
&.&
 \left(q_v\|\varphi\|_{L^2([x_{j+\frac{3}{4}},x_{j+\frac{5}{4}}])}\|v\|_{L^2(I_{j+\frac{1}{2}})}+(q_u-1)\|\varphi\|_{L^2(I_j)}\|v\|_{L^2([x_{j+\frac{1}{4}},x_{j+\frac{3}{4}}])}\right).
\end{eqnarray*}
Similarly, we obtain 
\begin{equation*}
    \Theta_{v,1} \leq \frac{c}{h} 2 \hat{C}_{3,\frac{1}{2}} \sum_j \left((q_u-1)\| cw \|_{L^2(I_j)}\|\psi\|_{L^2(I_j)}
     + q_v\| cw \|_{L^2(I_{j+\frac{1}{2}})}
     \|\psi\|_{L^2(I_{j+\frac{1}{2}})}\right).
\end{equation*}
By applying Cauchy-Schwarz inequality, we have
\begin{eqnarray}
 \Theta_{u,1}+\Theta_{v,1} \leq \hspace{3.0in} & & \label{eq:Tuv1} \\     \frac{c}{h} C_{3,\frac{1}{2}} (q_u+q_v-1) \left(\| cw \|_{L^2(\Omega)}^2+\|v\|_{L^2(\Omega)}^2\right)^{1/2}\left(\| c \varphi \|_{L^2(\Omega)}^2+\|\psi\|_{L^2(\Omega)}^2\right)^{1/2}. & & \nonumber 
\end{eqnarray}
Next, we bound the boundary terms $\Theta_{u,2},\Theta_{v,2}$. By using the second  inverse inequality in \eqref{eq:inv2} with a linear scaling from $[-1, 1]$ to $I_j$ (or to $I_{j-\frac{1}{2}}$), we reach
\begin{align*}
\Theta_{u,2} 
&\leq \frac{c}{h} 2 \hat{C}_{4,\frac{1}{2}}^2\sqrt{q_u(q_v+1)}\sum_j\| c \varphi \|_{L^2(I_j)}\Big(\|v\|_{L^2(I_{j-\frac{1}{2}})}+
\|v\|_{L^2(I_{j+\frac{1}{2}})}\Big),\\
\Theta_{v,2} &\leq  \frac{c}{h} 2 \hat{C}_{4,\frac{1}{2}}^2\sqrt{q_u(q_v+1)}  \sum_j
 \Big(\| cw \|_{L^2(I_j)}\|+
\| cw \|_{L^2(I_{j+1})}\Big)\|\psi\|_{L^2(I_{j+\frac{1}{2}})},
\end{align*}
and hence
\begin{eqnarray}
 \Theta_{u,2}+\Theta_{v,2} \leq \hspace{3.0in} & & \label{eq:Tuv2} \\
 \frac{c}{h}C_{4,\frac{1}{2}}\sqrt{q_u(q_v+1)} \left(\| cw \|_{L^2(\Omega)}^2+\|v\|_{L^2(\Omega)}^2\right)^{1/2}\left(\| c \varphi \|_{L^2(\Omega)}^2+\|\psi\|_{L^2(\Omega)}^2\right)^{1/2}. & & \nonumber
\end{eqnarray}
For $\Theta_{u,3}$ and  $\Theta_{v,3}$, using the inverse inequality in \eqref{eq:inv1} with a linear scaling from $[-1, 1]$ to $I_s$ (with $s=j, j\pm 1, j\pm\frac{1}{2}, j+\frac{2}{3}$),
\begin{align*}
\Theta_{u,3}&
\leq \frac{c}{h} 2 \hat{C}_2^2 c \beta   q_u^2 \sum_j \| c \varphi \|_{L^2(I_j)}
 \Big(2\| cw \|_{L^2(I_j)}
 + \| cw \|_{L^2(I_{j-1})} + \| cw \|_{L^2(I_{j+1})}\Big),\\
\Theta_{v,3} &
\leq  \frac{c}{h} 2 \hat{C}_2^2 \f {\tau}{c} (q_v+1)^2 \sum_j\|\psi\|_{L^2(I_{j+\frac{1}{2}})}
\Big(2\|v\|_{L^2(I_{j+\frac{1}{2}})}
+ \|v\|_{L^2(I_{j-\frac{1}{2}})} + \|v\|_{L^2(I_{j+\frac{3}{2}})}\Big),
\end{align*}
and hence
\begin{eqnarray}
 \Theta_{u,3}+\Theta_{v,3} \leq \hspace{3.0in} & & \label{eq:Tuv3} \\   \frac{c}{h} C_5 \max\{c \beta   q_u^2, \frac {\tau}{c} (q_v+1)^2\} 
 \left(\| cw \|_{L^2(\Omega)}^2+\|v\|_{L^2(\Omega)}^2\right)^{1/2}\left(\| c \varphi \|_{L^2(\Omega)}^2+\|\psi\|_{L^2(\Omega)}^2\right)^{1/2}. & & \nonumber
\end{eqnarray}

Finally, we combine \eqref{eq:Tuv0}, \eqref{eq:Tuv1}-\eqref{eq:Tuv3}, and conclude 
\begin{equation*}
 \|\mathcal{L}_c\| \leq
    \frac{c}{h}C_{3,\frac{1}{2}} (q_u+q_v-1)
   +\frac{c}{h}C_{4,\frac{1}{2}} \sqrt{q_u(q_v+1)}+\frac{c}{h}C_5\max\left\{ c \beta   q_u^2, \frac{\tau}{c} (q_v+1)^2\right\}.
   \end{equation*}
\end{proof}

\section{Timestepping}\label{sec:timestepping}
After the spatial semi-discretization, 
we are faced with evolving the linear system of equations 
\begin{equation}\label{local_ode_5}
M \frac{d W^h}{dt} = AW^h + F^h.
\end{equation}
Here $M$ and $A$ are the mass and stiffness matrices corresponding to the spatial discretizations at hand and $W^h$ is a vector containing all degrees of freedom.  To exploit the operator bounds \eqref{bounds_cdg} we note that we can partition $W^h$,
\begin{displaymath}
W^h = \left( \ba{c} W_1^h \\ W_0^h \ea \right) ,
\end{displaymath}
where $W_1^h$ includes all degrees-of-freedom except the cell averages of $u^h$ and $W_0^h$ simply consists of those cell averages. In other words, in each element we write
\begin{displaymath}
u^h = u_1^h + u_0^h, \ \ \int_{\Omega_j} u_1^h dx =0. 
\end{displaymath}
Similarly partitioning the test functions we see that the semi-discrete system is of the form
\begin{equation}\label{partitioned_ode}
\left( \ba{cc} M_1 & 0 \\ 0 & M_0 \ea \right) \f {d}{dt} \left( \ba{c} W_1^h \\ W_0^h \ea \right) = \left( \ba{cc} A_{11} & 0 \\ A_{01} & 0 \ea \right) \left( \ba{c} W_1^h \\ W_0^h \ea \right) + \left( \ba{c} F_1^h \\ F_0^h \ea \right) .
\end{equation}

This structure implies that stability is determined by the time stepping scheme applied to the $W_1^h$ subsystem; $W_0^h$ can be computed independently via an integration in time, though in practice we have used the same scheme for all degrees-of-freedom. In addition, as the $W_0^h$ subsystem is simply the discrete form of \eqref{extra_s}, the matrix $M_0^{-1} A_{01}$ will be uniformly bounded in both $h$ and the polynomial degrees. The operator bounds derived in section \ref{sec:operatorbounds} directly apply to the $W_1^h$ subsystem under the restrictions
given there (one space dimension and periodic boundary conditions). In particular as the norm induced by $M_1$ is simply the sum of the $L^2$ norms of $v^h$ and $\f {\pa u^h}{\pa x}$ we have:
\begin{align}
\||a\||&=\textrm{sup}_{0\ne g, \phi\in F_h}\frac{a(g,\phi)}{\|g\|_{L^2(\Omega)}\|\phi\|_{L^2(\Omega)}}\notag\\
&=\textrm{sup}_{{\bf x}, {\bf y}\ne 0}\frac{{\bf x}^T A_{11} {\bf y}}{\|{\bf x}\|_{M_1}\|{\bf y}\|_{M_1}}
=\textrm{sup}_{{\bf x}, {\bf y} \ne 0}\frac{\langle {\bf x}, M_1^{-1} A_{11} {\bf y}\rangle_{M_1}}{\|{\bf x}\|_{M_1}\|{\bf y}\|_{M_1}}=\|M_1^{-1} A_{11}\|_{M_1}.
\end{align}
Here $F_h=U_h^{q_u-1}\times \tilde{V}_h^{q_v}$, and the staggered method for the unknowns $f^h=(u_x^h, \tilde{v}^h)\in F_h$ can be written as 
$(\frac{d f^h}{dt}, \phi)=a(f^h, \phi) \; \forall \phi\in F_h$. 
The simplest application of this result is in the case of central fluxes, which is what we use in the numerical experiments. Then $A_{11}$ is skew-symmetric and therefore the generalized eigenvalue problem $i \alpha_j M_{11} \phi_j = A_{11} \phi_j$ has orthonormal eigenvectors in the induced inner product; thus a simple von Neumann analysis applies. In fact for the central fluxes we can prove that stability-enhanced leap-frog schemes as constructed in \cite{JolyRodriguezLeapFrog} of the same order as the spatial discretization and a number of stages proportional to the order can always be used with a CFL number $\mathcal{O}(1)$. We note that optimized schemes are constructed in \cite{JolyRodriguezLeapFrog} and in Figure \ref{fig:spectrum_1d_periodic} we display
time step stability limits based on these. 

\begin{truth}\label{LeapFrog}
Under the assumptions of Theorem \ref{thm:staggerbound} and using central fluxes, there exist constants $C_1$ and $C_2$ independent of $q_u$ and $q_v$ and time stepping schemes with order
$q_T \geq \max (q_u,q_v)$ and a number of stages bounded by $C_2 q_T$ such that the fully discrete method is stable under the CFL condition $c \Delta t \leq C_1 h$.
\end{truth}

\begin{proof}
The results in \cite{JolyRodriguezLeapFrog} may be directly applied to the second order equation $M_1 \f {d^2 W_1^h}{dt^2}=A_{11}M_1^{-1} A_{11} W_1^h$. In particular they show that order $q_T$ leap-frog schemes with $q_T$ stages (applications of the spatial operator) can be constructed which are stable for $\Delta t^2 \leq \f {q_T^2}{e^2 \rho^2(M_1^{-1} A_{11})}$ with $\rho$ denoting the spectral radius. This establishes the result. It is possible to adapt their arguments to leap-frog schemes applied to the first order case, leading to the analogous inequality $\Delta t \leq \f {q_T}{e \ \rho(M_1^{-1} A_{11})}$, but we omit the algebraic details.
\end{proof}

More generally, we can invoke the Kreiss-Wu theory \cite{kreiss1993stability} to relate the time step stability limits to the local stability radius of any locally stable time stepping we employ. This theory is based on energy estimates which we have derived above. In particular if the local stability radius is $R$ the fully discrete method is stable if $\Delta t < \f {R}{\| M_1^{-1} A_{11} \|_{M_1}} = O( h/\max(q_u,q_v))$. In our numerical experiments we simply use Taylor time stepping. Given the value of $W^h$ at time $t$ \begin{equation}\label{taylor_5}
W^h(t) \approx \sum_{j=0}^{q_{\rm T}} \f {(t-t_n)^j}{j!} \f {d^j W^h}{dt^j}
\end{equation}
This expansion can easily be computed as time derivatives of $W^h$ at $t = t_n$,  and  can be obtained sequentially by (\ref{local_ode_5}) and the time step is completed by setting $t=t_n + \Delta t$. As is well-known, the Taylor methods are locally stable for orders $q_{\rm T}=3,4,7,8,11,12, \ldots$. However, for $q_{\rm T}$ large they satisfy $R \sim \pi$ for $q_T$ a multiple of $4$
and $R \sim \pi/2$ for $q_{\rm T}$ one less than a multiple of $4$ \cite{ketcheson2015absolute}. Thus we cannot use them with an order-independent CFL number. (We conjecture that
stability-enhanced schemes can be built off of the Taylor methods and have some preliminary examples, but they are not used here.) Despite this we find that if the spatial discretization
order is bounded by 16 we can march at the same or greater order with a CFL number bounded by 0.1. We emphasize that the method is not restricted to the stability-enhanced leap-frog schemes or the Taylor-based methods used in the experiments. Any standard locally-stable scheme can be used for all our choices of numerical flux and reversible schemes can be applied for the energy-conserving fluxes. 

On the other hand, when physical boundaries are present the mesh cannot be staggered and the time steps must be reduced to maintain stability. Fortunately this step size reduction can be localized to a few elements near the boundary. Following the local time stepping method by Diaz and Grote \cite{diaz2009energy} we advance the solution for one time step $\Delta t$, starting with $W^h(t_n)$, as follows.

\begin{enumerate}
\item Partition $W^h$ into two parts, $W^h_B$ consisting of degrees-of-freedom associated with
elements near the boundary, and $W^h_I$.
\item Compute all terms $\f {d^{\ell} W^h (t_n)}{dt^{\ell}}$, $\ell=1, \ldots , q_{\rm T}$.
These can be used to update $W^h_I (t_n+\Delta t)$.
\item To update $W_B^h$ take $p$ sub-steps with step size $\delta t = \Delta t / p$ using
(\ref{local_ode_5})-(\ref{taylor_5}) with $W^h$ replaced by $W_B^h$. Here flux terms associated with the interface between elements assigned to the near boundary group and the interior group
give rise to a forcing function $F_B^h$. We evaluate $F_B^h$ and all necessary derivatives at
any intermediate time step using (\ref{taylor_5}).
\end{enumerate}

\section{Numerical Experiments}\label{sec:experiment}

In this section, we present numerical experiments that illustrate the properties of our staggered method. In all cases we use a modal formulation with tensor product Legendre polynomials and we use exact integration (through the use of quadrature of sufficiently high order) to compute the integrals in the variational formulation. For all tests, we use purely central fluxes, i.e. we set  $\tau,\beta   = 0$. 

\subsection{Computed Rates of Convergence}

\begin{figure}[ht]
  \begin{center} 
\includegraphics[width=0.49\textwidth]{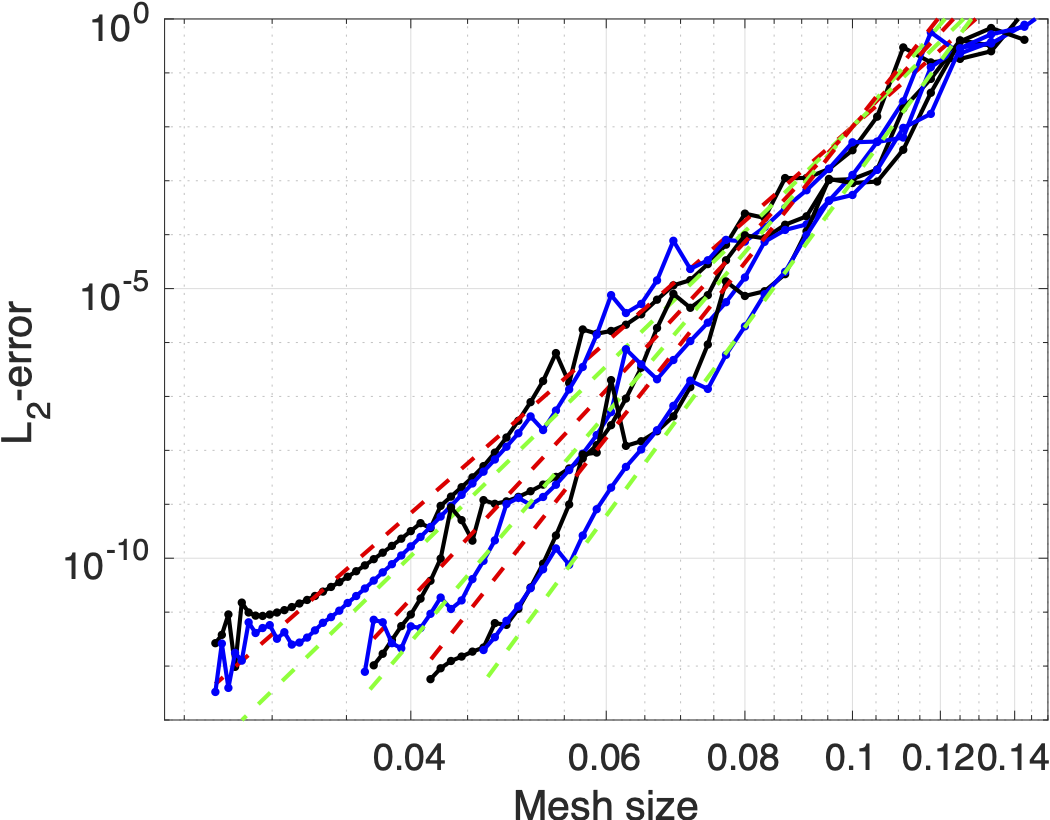}
  \includegraphics[width=0.49\textwidth]{l2_error_sin_cos_181922232627} \caption{To the left are the $L_2$-errors in $u^h$ for $q_u = 2, 3, 6,7,10,11,14,15$ and to the right for $q_u = 18,19,22,23,26,27$. The dashed lines have slope $q_u$ when $q_u$ is even and $q_u+1$ when $q_u$ is odd corresponding to the expected rates of convergence for central fluxes.  \label{fig:L2_errors_sin_cos}}
 \end{center}
\end{figure}
Here we evolve the exact solution  
\[
u(x,t) = \sin(\omega(x+t)), \ \  v(x,t) = \omega \cos (\omega(x+t)), 
\]
on the periodic domain $\Omega=[-1,1]$ until $T = 2.2$. We discretize using the staggered scheme with $q_u = 2,3,6,7,10,11,14,15,18,19,22,23,26,27$ and $q_v = q_u-1$. In order to make it possible to observe the rates of convergence we set $\omega = 2 q_u \pi$ for $q_u = 2,3,6,7,10,11,14,15$ and $\omega = 4 q_u \pi$ for $q_u = 18,19,22,23,26,27$. 

To evolve in time we use Taylor series time stepping with $q_{\rm T}=q_u+1$ (the stability domains of all of these Taylor series methods contain the imaginary axis) and throughout we keep the ratio $\frac{\Delta t}{h} = 0.1$. The $L_2$-errors in the solution $u^h$ as a function of the element size $h$ are displayed in Figure \ref{fig:L2_errors_sin_cos}. As can be seen from the figure the rates of convergence (as indicated by the dashed lines) appear to be optimal,  i.e. $q_u+1$, when $q_u = 3,7,\ldots$ and suboptimal by one, i.e. $q_u$ when $q_u = 2,6,\ldots$. This is consistent with the analysis and numerical experiments for the non-staggered scheme and central fluxes; see \cite{Upwind2}.

\subsection{Spectral Radii of Periodic Semi-discretization}

Consider now the matrix, $A$, in the semi-discretization (\ref{local_ode_5}).
With purely central fluxes, the eigenvalues of $A$ will be  imaginary and based on the estimates on the operator norm of $\mathcal{L}_c$ we expect them to grow linearly with $q_u$. In this experiment we set $q_v = q_u$ and consider a computational domain $\Omega = [-1,1]$. 

In Figure \ref{fig:spectrum_1d_periodic} we display the spectral radii of the matrix $A$, i.e. the eigenvalue of $A$ with the largest magnitude, $\lambda_\infty$, scaled by $(h/q_u)$ for three different element sizes $h = 2/5, 2/10, 2/20$.  As can be seen the growth of the spectral radii appears to be asymptotically linear in $q_u$ (i.e. constant when scaled by $q_u^{-1}$). The left figure displays ratio of the square root of the diagonal entries in Fig. 2 in \cite{JolyRodriguezLeapFrog} and the  spectral radii of the time stepping matrix $A$ scaled by the element size $h$. Note that the enhanced stability limits  ($\sqrt{\alpha_{m,k}}, k = m-1$) given in \cite{JolyRodriguezLeapFrog} are only available for even orders so for $q_u = 3$ we use the 4th order limit and for $q_u=5$ we use the 6th order limit, etc. From the figure we see that the ratio (which corresponds to the CFL number) is at least 0.6 for all orders considered. Also note that if such stability enhanced methods were available for the leap-frog scheme (which, again, we believe is possible) they would be particularly efficient for $q_u=3$ and 8. 
\begin{figure}[ht]
\begin{center} 
\includegraphics[width=0.48\textwidth]{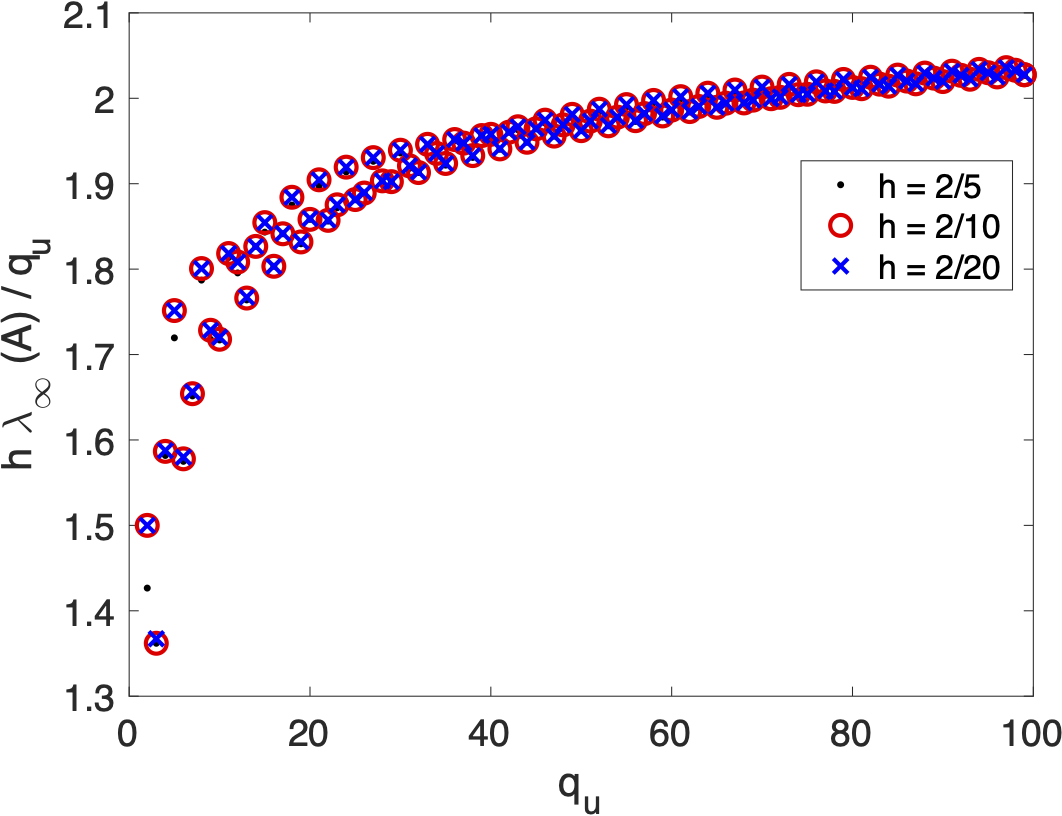}
\includegraphics[width=0.48\textwidth]{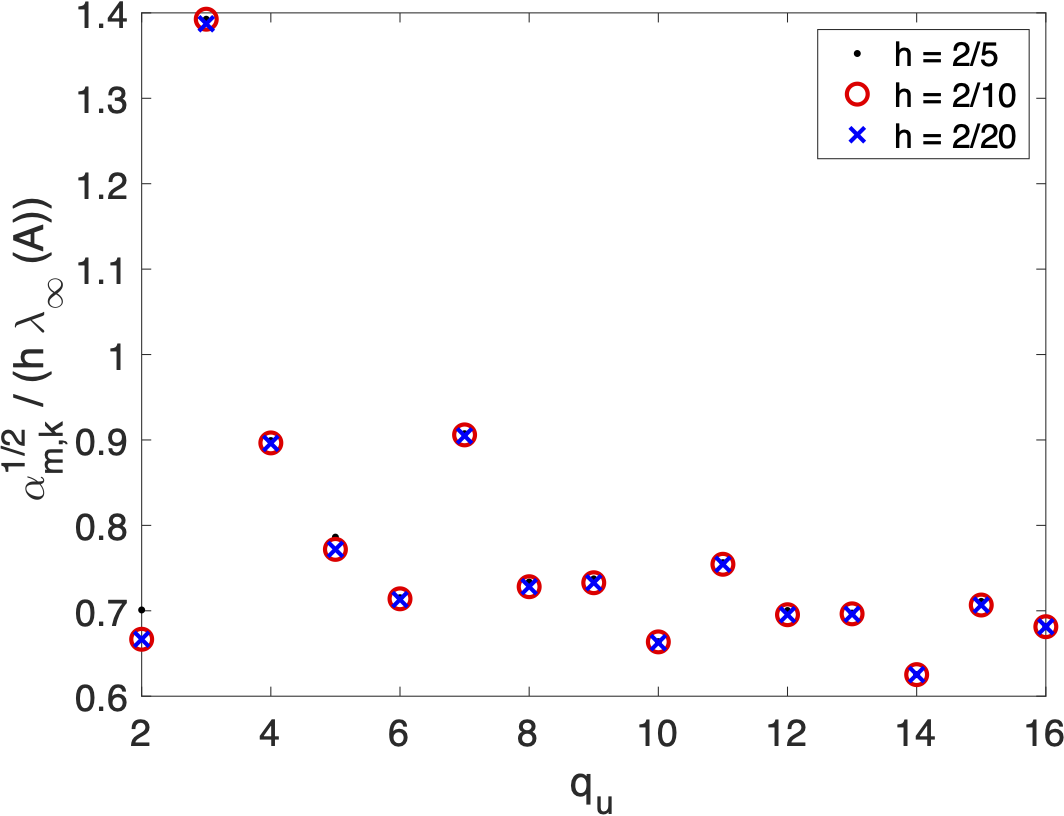}
\caption{The right figure displays the spectral radii of the time stepping matrix $A$ scaled by the element size $h$ and the reciprocal of $q_u$ as a function of $q_u$. The left figure displays ratio of the square root of the diagonal entries in Fig. 2 in \cite{JolyRodriguezLeapFrog} and the the spectral radii of the time stepping matrix $A$ scaled by the element size $h$. See the text for details.\label{fig:spectrum_1d_periodic}}
\end{center}
\end{figure}

\subsubsection{Numerical Investigation of Stability of the Local Timestepping}
\begin{figure}[ht]
	\begin{center} 
		\includegraphics[width=0.49\textwidth]{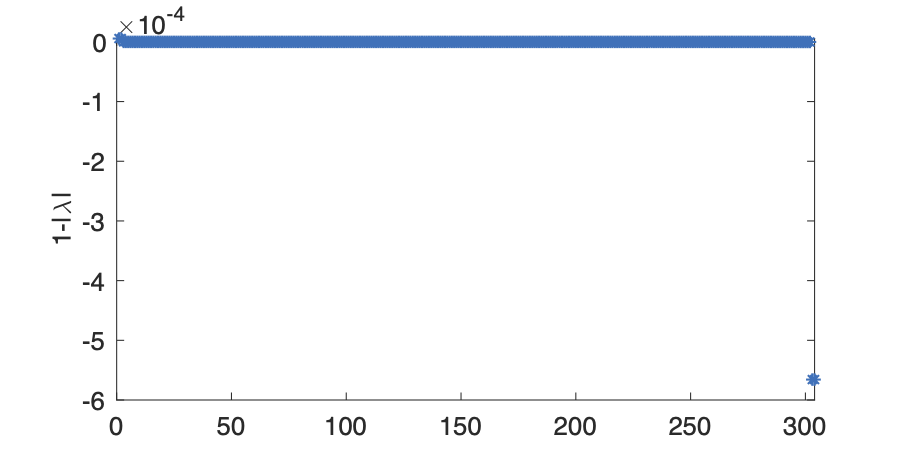}
		\includegraphics[width=0.49\textwidth]{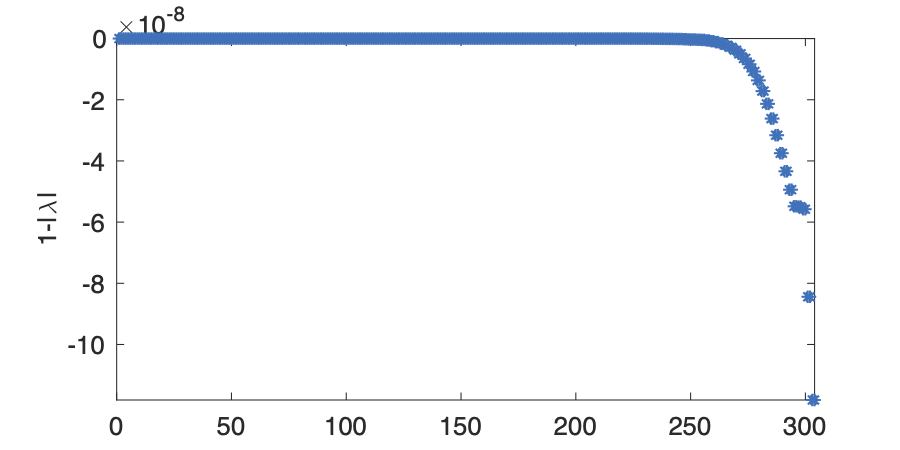} \caption{On the left $1-\arrowvert \lambda_j \arrowvert$, $\{ \lambda_j \}$ the eigenvalues of $B$ from (\ref{onestep}) for $q_u = 14$ and $q_v = 13$ with $m = 2$; to the right $1-\arrowvert \lambda_j \arrowvert$ for $q_u = 14$ with $m = 3$. \label{fig:s_bc_dg}}
	\end{center}
\end{figure}
In this section, the computational domain is chosen  to be $[-1,1.5]$. We impose a homogeneous Neumann boundary condition at the left boundary  and a homogeneous Dirichlet boundary condition at the right boundary. The discretization is carried out on a staggered uniform mesh with mesh size $h$. The process of evolving the solution a full timestep by the local timestepping procedure described above can be expressed as a matrix multiplication 
\begin{equation}
W^h(t_{n+1}) = BW^h(t_{n}). \label{onestep}
\end{equation}
Here, again, $W^h$ is a vector containing the modes describing the element-wise expansions of the displacement and the velocity. The eigenvalues $\lambda$ of the matrix $B$ reveal if a particular discretization is stable. As we use a central flux all the eigenvalues should satisfy $|\lambda| = 1$. In practice, the accuracy of the eigenvalue computation can make it difficult to distinguish if the eigenvalues are strictly smaller than one, equal to one, or slightly larger than one. If the largest eigenvalue is slightly larger than one, say $|\lambda | = 1+\delta$, this may be an indication of an unstable method. However, if $\delta$ is very small and does not change as the mesh is refined the method may still be considered useful even though it cannot be claimed to be stable in a mathematically strict sense.       

We have found that for very high degrees and when the local timestepping is used, the thickness, $m$, of the layer where the local timestepping is used can impact the size of $\delta$.  In this experiment we always set the parameters of the local timestepping as $q_{\rm T} = p = q_u + 1$. 

We first fix the number of DG elements for $u$ to be $10$, i.e., $h = 2.5/10$, and the number of DG elements for $v$ is $11$. The degrees of the approximation spaces for $u$ and $v$ are chosen to be $q_u = 14$ and $q_v = 13$, respectively. The ratio between time step size $\Delta t$ and the mesh size $h$ are fixed as $\frac{\Delta t}{h} = 0.1$.  

In Figure \ref{fig:s_bc_dg} we display $1-|\lambda| = - \delta$ as a function of the eigenvalue index. The left figure is for an overlap with $m=2$. We observe that the modulus of the largest eigenvalue is larger than $1$ by about $6\cdot 10^{-4}$. This would correspond to a magnification of about 2 of an unstable mode after about 1400 time steps, indicating a fast growing instability. The right figure displays the same method except that the overlap is now increased to $m= 3$. Now we find that the modulus of the largest eigenvalue is larger than $1$ by about $10^{-7}$. As this means that it will take around 7 million time steps before this mode is doubled in amplitude it is unlikely that it would show up in any practical computation.  

Importantly, $\delta$ appears to be robust to grid refinement. In Figure \ref{fig:s_bc_dn}, we fix the $m = 3$ and increase the number of DG elements for $u$ from $20$ to 40 and $80$. Again we find that the modulus of the largest eigenvalue is larger than $1$ by about $10^{-7}$ for all three discretizations. 
\begin{figure}[htb]
	\begin{center} 
		\includegraphics[width=0.32\textwidth]{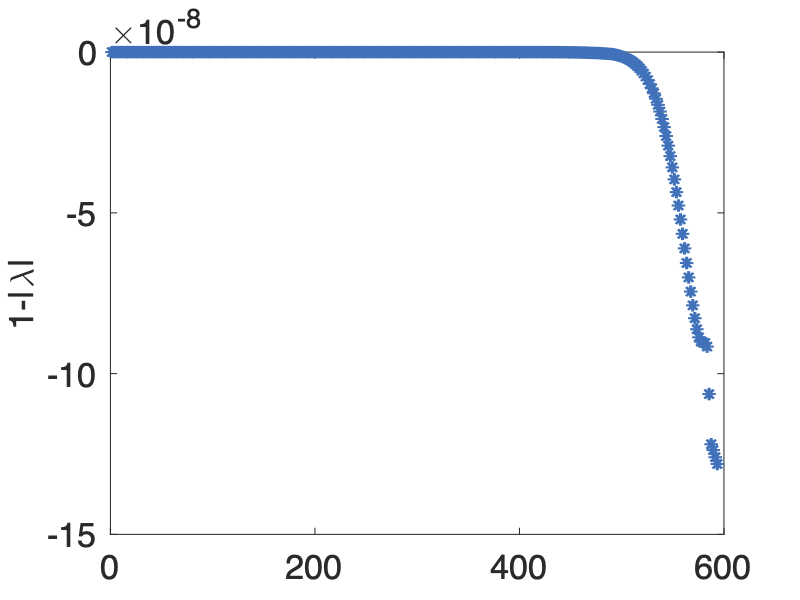}
		\includegraphics[width=0.32\textwidth]{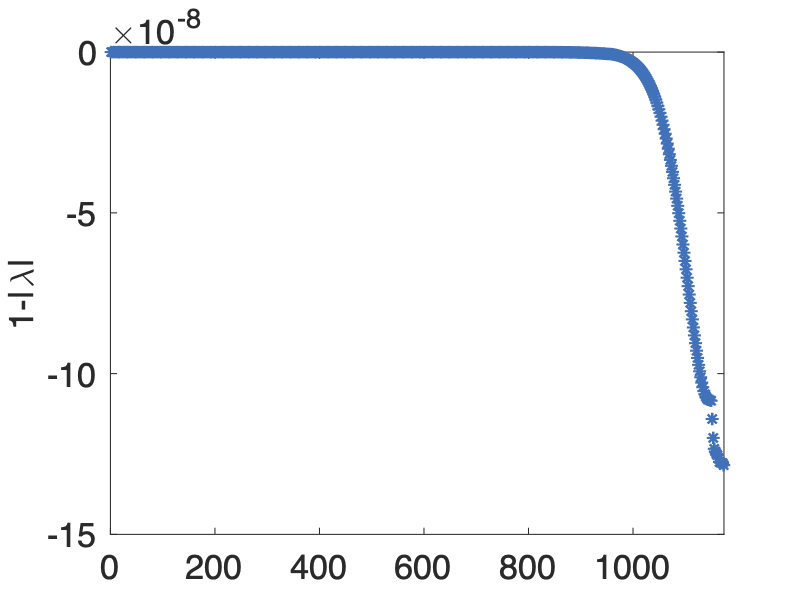}
		\includegraphics[width=0.32\textwidth]{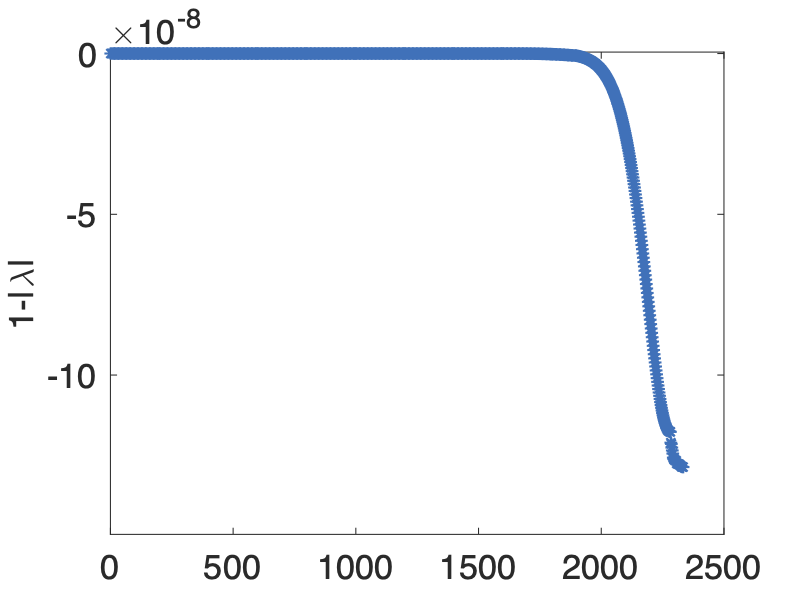} \caption{From left to right are $ 1- \arrowvert \lambda_j \arrowvert$, $\{ \lambda_j \}$ the eigenvalues of $B$ in (\ref{onestep}), for $q_u = 14$ and $q_v = 13$ with $m = 3$ and the number of elements in $\Omega_h$ are $n = 20,40,80$, respectively. \label{fig:s_bc_dn}}
	\end{center}
\end{figure}

\subsection{Convergence in two dimensions with Dirichlet boundary condition}\label{convergence_2d}

In this section, we investigate the convergence of the staggered energy-based DG scheme with the local time stepping of section \ref{sec:timestepping} and variable sound wave speed $c(x,y)$ in two space dimensions. Precisely we solve
\[
\frac{\partial^2 u}{\partial t^2} = \nabla\cdot(c^2(x,y) \nabla u) + f(x,y,t), \ \ \ (x,y)\in[-1,1]\times[-1,1],\ \ \ t>0,
\]
where $c(x,y) = 1 + x^2 + y^2$. Further, we construct a manufactured solution so that 
\begin{align*}
u(x,y,t) &= \sin(\sqrt{k_1^2+k_2^2}\pi t)\sin(k_1\pi x)\sin(k_2\pi y), \\
v(x,y,t) &= \sqrt{k_1^2+k_2^2}\pi\cos(\sqrt{k_1^2+k_2^2}\pi t)\sin(k_1\pi x)\sin(k_2\pi y).
\end{align*}
That is, the initial condition and the external forcing function $f(x,y,t)$ are determined by this manufactured solution. The boundary conditions are homogeneous Dirichlet conditions. To allow for sufficient range to compute the errors we set $k_1 = k_2 = q = 2$  for $q_u = q_v = q = 2,3$, and $k_1 = k_2 = 2q$ for $q_u = q_v = q = 6,7$ with $q$ being the degree of the approximation space for both $u$ and $v$.
\begin{figure}[ht]
\begin{center} 
	\includegraphics[width=0.45\textwidth]{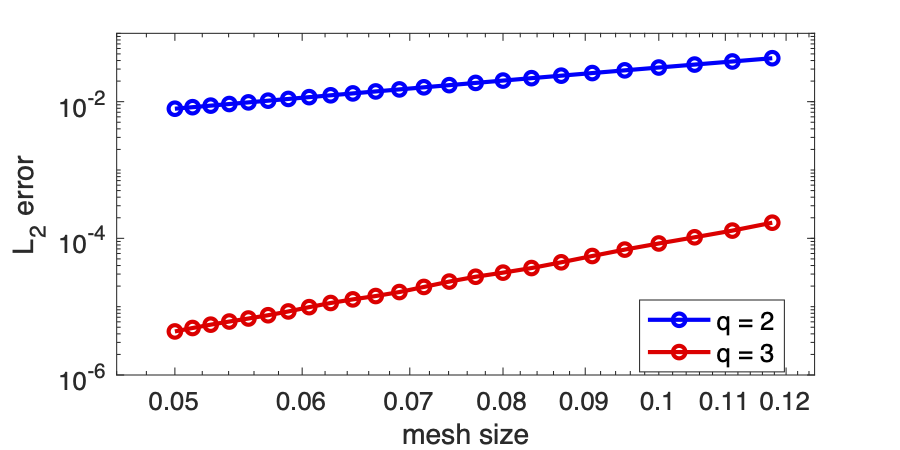}
	\includegraphics[width=0.45\textwidth]{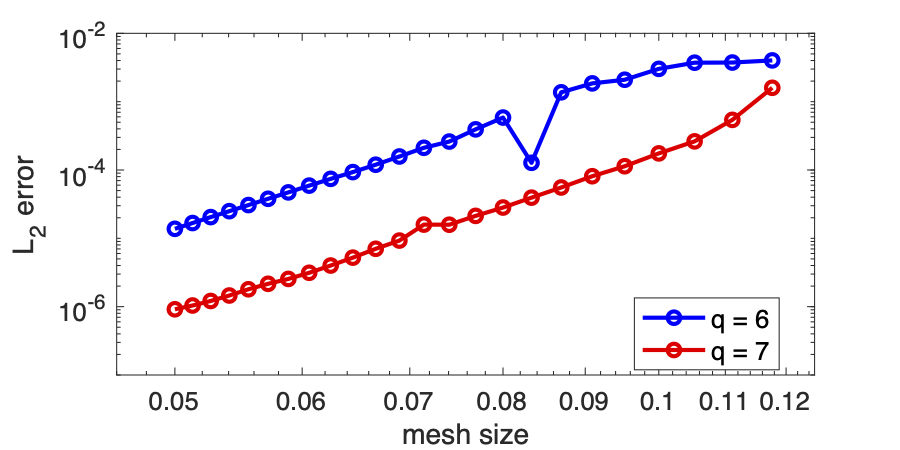}
	\caption{The $L^2$ errors for $u$, from left to right, are for  $q_u = q_v = 2,3$ and $q_u = q_v = 6,7$, respectively. \label{fig:urate_2d}}
\end{center}
\end{figure}

The discretization is performed with staggered elements. The mesh $\Omega^h$ corresponding to the piecewise polynomial approximation to $u$ is Cartesian with vertices given by $x_i = ih$, $y_j = jh$, $i, j = 0,1,\cdots ,n$ with $h = 2/n$. In the interior the elements of $\Omega^{\diamond,h}$ corresponding to the piecewise polynomial approximation to $v$ are staggered with respect to $\Omega^h$; its vertices are $x_{i+1/2}=(i+1/2)h$, $y_{j+1/2}=(j+1/2)h$ . Near the boundaries the elements for $v$ are reduced in size by a factor of 1/2 or 1/4. Then we have $n^2$ elements in $\Omega^h$ and $(n+1)^2$  elements in $\Omega^{\diamond,h}$. Figure \ref{fig:staggered_domain_2d} gives an illustration of the staggered grids with $n = 3$.
\begin{figure}[ht]
\begin{center} 
\includegraphics[width=0.7\textwidth,trim={0.1cm 1.2cm 0.55cm 0.45cm},clip]{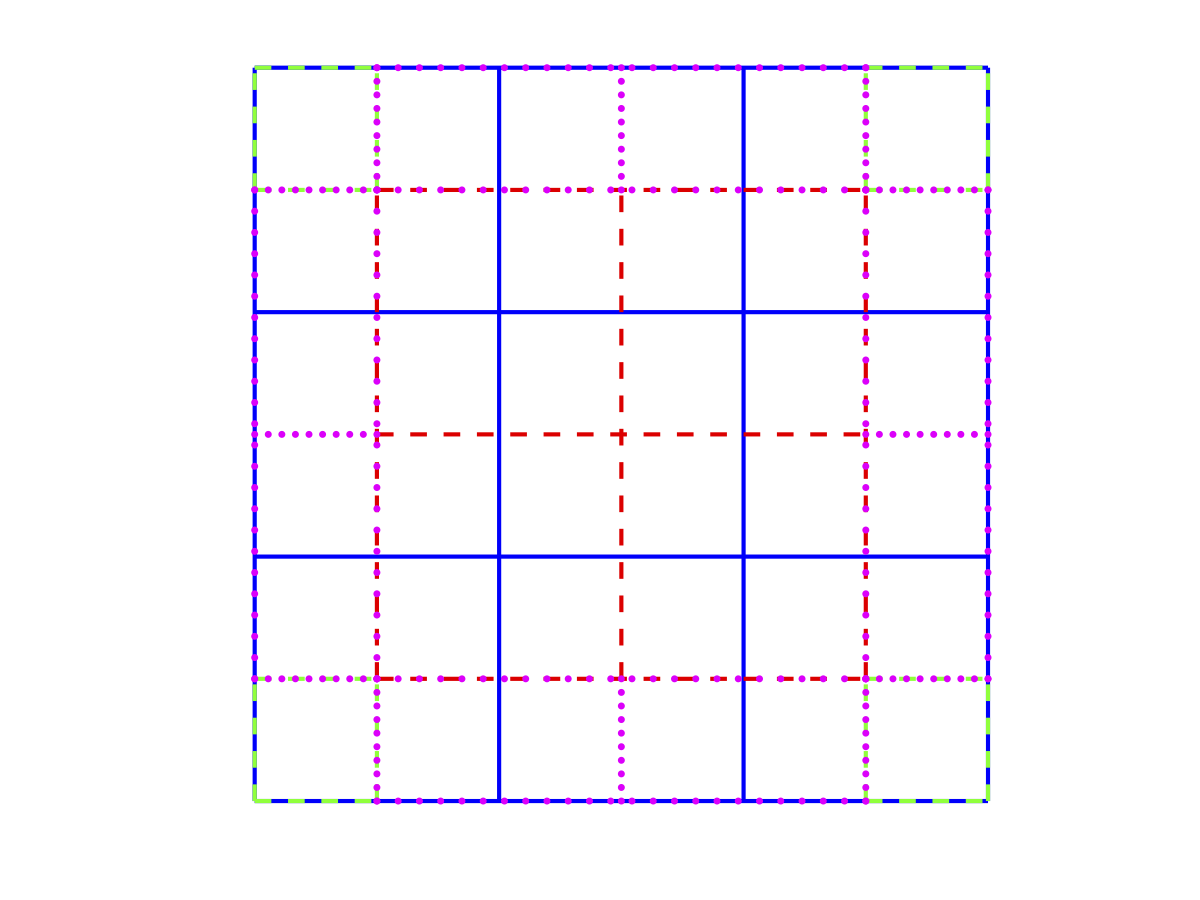}\caption{A staggered grid in two dimensions. Blue boxes are  elements of $\Omega^h$ corresponding to the piecewise approximation to $u$ and non-blue boxes are elements of $\Omega^{\diamond,h}$ corresponding to the piecewise approximation to $v$. Here, we have $3\times3 = 9$ elements in $\Omega^h$ and 4 interior (red boxes), 8 edge (magenta boxes) and 4 corner (green boxes) elements in $\Omega^{\diamond,h}$.  \label{fig:staggered_domain_2d}}
\end{center}
\end{figure}

Here we use the central flux, $\beta=\tau=0$.
We evolve the solution by the local Taylor time stepping described in section \ref{sec:timestepping} with $p=q_{\rm T} = q+1$ and $m = 3$ until the final time $T = 0.5$. The ratios of the time step size $\Delta t$ and mesh size $h$ are set to be $\frac{\Delta t}{h} = 0.1$. 

\begin{table}[htb]
\begin{center}
\begin{tabular}{|c|c c c c |}
\hline
{Degree ($q$) of approx. to $u$ }  & 2 & 3 & 6 & 7  \\
\hline
{Rate fit with C.-flux}&2.00 & 4.27 & 7.21 & 8.13 \\
\cline{1-5}
\end{tabular}
\end{center}
\caption{Linear regression estimates of the convergence rate for $u$ with central flux in two dimensions. The degree of the approximation space for $u$ and $v$ are $q$ for both $x$ and $y$ directions.}
\label{convergence_rate_u_2d}
\end{table}

The $L^2$ errors for $u$ are plotted against the mesh size $h$ in Figure \ref{fig:urate_2d}. Table \ref{convergence_rate_u_2d} presents the linear regression estimates of the convergence rate for $u$ based on the data in Figure \ref{fig:urate_2d}. From Table \ref{convergence_rate_u_2d}, we observe an optimal convergence rate of $q + 1$ when $q = 3,6,7$ and a suboptimal convergence by one for $q = 2$.  

\subsection{Numerical Investigation of the Stability of the Local Time Stepping in Two Dimensions}
\begin{figure}[ht]
	\begin{center} 
		\includegraphics[width=0.45\textwidth]{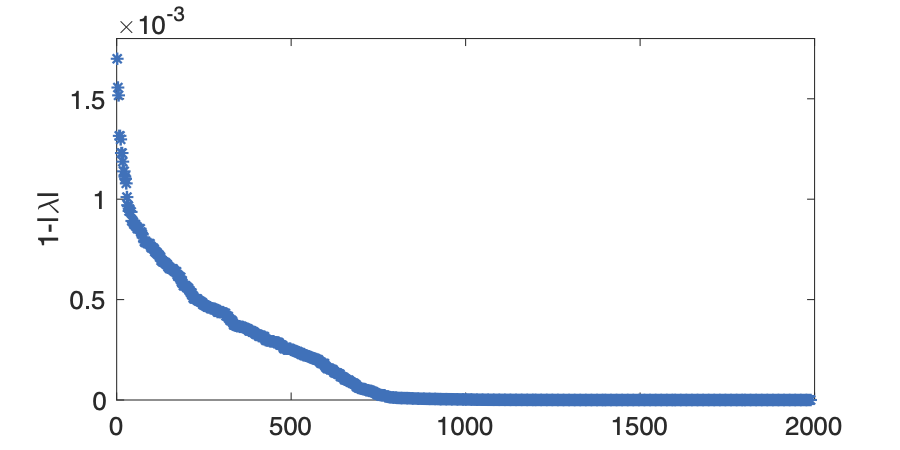}
		\includegraphics[width=0.45\textwidth]{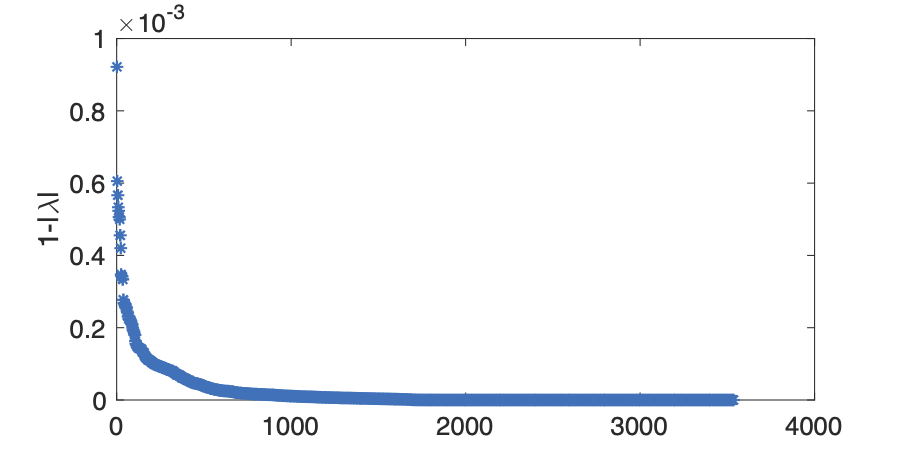}\\
		\includegraphics[width=0.45\textwidth]{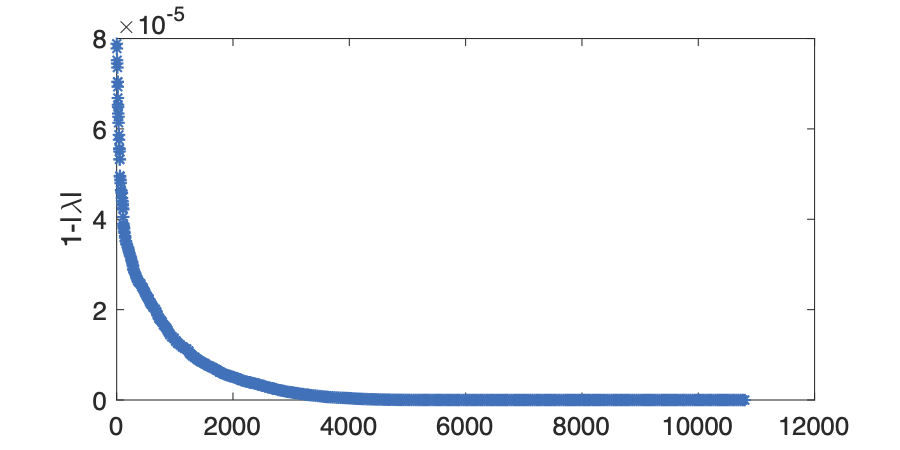}
		\includegraphics[width=0.45\textwidth]{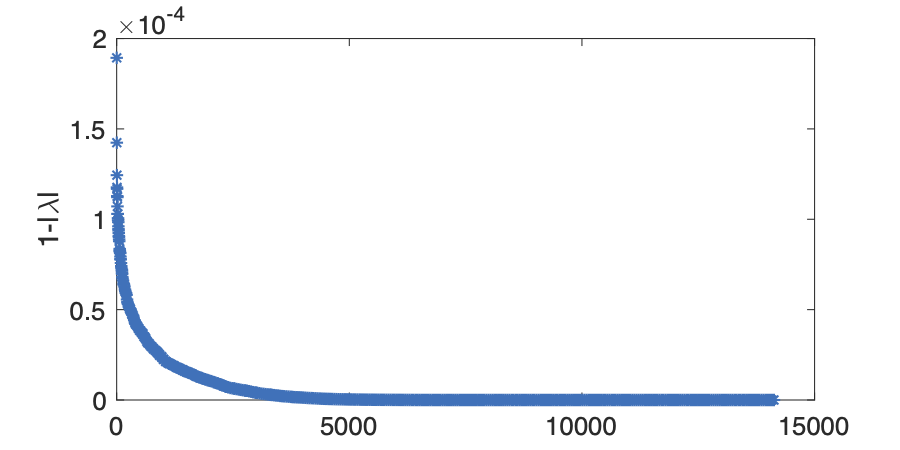} \caption{The top panel, from left to right are $ 1- \arrowvert \lambda_j \arrowvert$, $\{ \lambda_j \}$ the eigenvalues of $B$ in (\ref{onestep}), for the two-dimensional case with $q_u = q_v=2,3$, respectively. The bottom panel displays the same results for $q_u=q_v=6,7$, respectively. \label{fig:s_bc_dq_2d}}
	\end{center}
\end{figure}

In this section, we investigate the stability of the full discretization with the local time stepping in two dimensions with homogeneous Dirichlet boundary conditions. The computational domain is chosen to be the same as above, so is the spatial discretization. Here, we set $n = 10$. Then, the number of elements in $\Omega^h$ is $100$ and in $\Omega^{\diamond,h}$ is $121$. The degree of the approximation space for $u$ and $v$ is $q$ for both. The parameters $p$, $q_{\rm T}$ in the local Taylor time stepping are set to be $p = q_{\rm T} = q+1$ and the overlap is set to $m = 3$.

In Figure \ref{fig:s_bc_dq_2d}, we display $1- | \lambda |$ for the fully discrete method with different values of $q$. The top panel, from left to right, displays the results for $q = 2,3$ and the bottom panel, from left to right displays the results for  $q = 6,7$. Here, we observe that $1- | \lambda |>0$ for all $q$ indicating that these particular discretizations are stable.

\section{Conclusion}

We have shown that, away from boundaries, the use of staggered meshes and suitably chosen numerical fluxes leads to energy-based DG methods for the wave equation with favorable time step stability bounds at high order. In particular, using explicit single step methods built from Taylor polynomials with degrees $q_{\rm T}=4s$, or $q_{\rm T}=4s-1$ and spatial approximations of comparable order we can stably march in time at a fixed, order-independent CFL number. A large global time step can be maintained if local time stepping is used near boundaries. Here we only consider simple geometries, but with local time stepping the proposed method should be applicable in more complex domains containing a sufficiently large volume separated from boundaries. 
\bibliographystyle{plain}
\bibliography{appelo,cdg}

\begin{thebibliography}{10}

\bibitem{Upwind2}
D.~Appel\"o and T.~Hagstrom.
\newblock A new discontinuous {G}alerkin formulation for wave equations in
  second order form.
\newblock {\em {SIAM} Journal On Numerical Analysis}, 53(6):2705--2726, 2015.

\bibitem{secondHermite}
D.~Appel\"{o}, T.~Hagstrom, and A.~Vargas.
\newblock Hermite methods for the scalar wave equation.
\newblock {\em SIAM Journal on Scientific Computing}, 40(6):A3902--A3927, 2018.

\bibitem{BANKS2016310}
J.W. Banks and T.~Hagstrom.
\newblock On {G}alerkin difference methods.
\newblock {\em Journal of Computational Physics}, 313:310--327, 2016.

\bibitem{cockburn1989tvb}
B.~Cockburn and C.-W. Shu.
\newblock {TVB} {R}unge-{K}utta local projection discontinuous {G}alerkin
  finite element method for conservation laws. {II.} {G}eneral framework.
\newblock {\em Mathematics of computation}, 52(186):411--435, 1989.

\bibitem{diaz2009energy}
J.~Diaz and M.~Grote.
\newblock Energy conserving explicit local time stepping for second-order wave
  equations.
\newblock {\em SIAM Journal on Scientific Computing}, 31(3):1985--2014, 2009.

\bibitem{HesthavenWarburton02}
J.~Hesthaven and T.~Warburton.
\newblock Nodal high-order methods on unstructured grids: I. time-domain
  solution of {M}axwell's equations.
\newblock {\em Journal of Computational Physics}, 181:186--221, 2002.

\bibitem{JolyRodriguezLeapFrog}
P.~Joly and J.~Rodr\'iguez.
\newblock Optimized higher order time discretization of second order hyperbolic
  problems: {C}onstruction and numerical study.
\newblock {\em Journal of Computational and Applied Mathematics},
  234:1953--1961, 2010.

\bibitem{ketcheson2015absolute}
D.~Ketcheson, L.~L{\'o}czi, and T.~Kocsis.
\newblock On the absolute stability regions corresponding to partial sums of
  the exponential function.
\newblock {\em IMA Journal of Numerical Analysis}, 35(3):1426--1455, 2015.

\bibitem{kreiss1993stability}
H.-O. Kreiss and L.~Wu.
\newblock On the stability definition of difference approximations for the
  initial boundary value problem.
\newblock {\em Applied Numerical Mathematics}, 12(1-3):213--227, 1993.

\bibitem{liu20082}
Y.~Liu, C.-W. Shu, E.~Tadmor, and M.~Zhang.
\newblock L2 stability analysis of the central discontinuous {G}alerkin method
  and a comparison between the central and regular discontinuous {G}alerkin
  methods.
\newblock {\em ESAIM: Mathematical Modelling and Numerical Analysis},
  42(4):593--607, 2008.

\bibitem{nessyahu1990non}
H.~Nessyahu and E.~Tadmor.
\newblock Non-oscillatory central differencing for hyperbolic conservation
  laws.
\newblock {\em Journal of computational physics}, 87(2):408--463, 1990.

\bibitem{fli}
M.~Reyna and F.~Li.
\newblock Operator bounds and time step conditions for the {DG} and central
  {DG} methods.
\newblock {\em Journal of Scientific Computing}, 62(2):532--554, 2015.

\bibitem{TAMECFL}
T.~Warburton and T.~Hagstrom.
\newblock Taming the {CFL} number for discontinuous {G}alerkin methods on
  structured meshes.
\newblock {\em {SIAM} Journal Numerical Analysis}, 46(6):3151--3180, 2008.

\end{thebibliography}
\end{document}